\documentclass{article}
\usepackage{geometry}
\usepackage{algorithm}
\usepackage{algpseudocode}
\usepackage{amsmath,amssymb,amsthm}
\usepackage{bm}
\usepackage{listings}
\usepackage{graphicx,float}
\usepackage{subfigure}
\usepackage{hyperref}
\usepackage{cleveref}
\usepackage{fancyhdr}
\usepackage{booktabs}
\numberwithin{equation}{section}
\newtheorem{assumption}{Assumption}[section]
\newtheorem{theorem}{Theorem}[section]
\newtheorem{proposition}{Proposition}[section]
\newtheorem{lemma}{Lemma}[section]
\newtheorem{remark}{Remark}[section]
\newtheorem{xmpl}{Example}[section]

\def\i{\mathrm{i}}
\def\d{\mathrm{d}}

\title{Frozen Gaussian sampling for scalar wave equations}
\author{Lihui Chai\thanks{School of Mathematics, Sun Yat-sen University, Guangzhou, 510275, China (chailihui@mail.sysu.edu.cn)} \and Ye Feng\thanks{Department of Philosophy, Sun Yat-sen University, Guangzhou, 510275, China (fengy78@mail2.sysu.edu.cn)} \and Zhennan Zhou\thanks{Beijing International Center for Mathematical Research, Peking University, Beijing, 100871, China (zhennan@bicmr.pku.edu.cn).}}
\date{}

\begin{document}
\maketitle
\begin{abstract}
    In this article, we introduce the frozen Gaussian sampling (FGS) algorithm to solve the scalar wave equation in the high-frequency regime. The FGS algorithm is a Monte Carlo sampling strategy based on the frozen Gaussian approximation, which greatly reduces the computation workload in the wave propagation and reconstruction. In this work, we propose feasible and detailed procedures to implement the FGS algorithm to approximate scalar wave equations with Gaussian initial conditions and WKB initial conditions respectively. For both initial data cases, we rigorously analyze the error of applying this algorithm to wave equations of dimensionality $d \geq 3$. In Gaussian initial data cases, we prove that the sampling error due to the Monte Carlo method is independent of the typical wave number. We also derive a quantitative bound of the sampling error in WKB initial data cases. Finally, we validate the performance of the FGS and the theoretical estimates about the sampling error through various numerical examples, which include using the FGS to solve wave equations with both Gaussian and WKB initial data of dimensionality $d = 1, 2$, and $3$.


\end{abstract}

\section{Introduction}
The wave equation is widely used in describing a variety of oscillation phenomena around an equilibrium. Computing solutions to wave equations plays an important role in fluid dynamics, seismology, electromagnetism, optics, relativistic physics, and so on.
In this paper, we consider the following wave equation,
\begin{equation}\label{wave_eq}
\partial_t^2 u - c^2(\bm{x})\Delta_{\bm{x}} u = 0,\quad \bm{x} \in\mathbb{R}^d, 
\end{equation}
with initial conditions,
\begin{equation}\label{init_con}
\begin{cases}
u(0,\bm{x}) = f_0^k(\bm{x}), \\
\partial_t u(0,\bm{x}) = f_1^k(\bm{x}),
\end{cases}
\end{equation}
where $u$ is the wavefield, $c(\bm{x})$ is the bounded velocity with its infimum $\mathrm{inf}_{\bm{x}\in\mathbb{R}^d} c(\bm{x}) > 0$, $d$ is the dimensionality, and $\Delta_{\bm{x}}$ is the Laplace operator in $\bm{x}$. We assume that the initial wavefield (\ref{init_con}) belongs to the Schwartz class, which is the class of rapidly decreasing functions on $\mathbb{R}^d$, and $k$ is the wave number or the spatial frequency of the wave. Basically, the number $k$ arises from the ratio of two different scales in the wave equation: the large length scale of the medium size, and the small wavelength scale. In the high-frequency regime, namely $k\gg 1$, the solution to \eqref{wave_eq} is highly oscillatory, which makes the numerical computation extremely hard, especially when $d$ is large. This is a manifestation of the well-known ``curse of dimensionality''. 

So far, two main categories of numerical methods have been developed and studied intensively in order to get the numerical solution accurately and efficiently. One category is direct methods based on full temporal-spacial discretization, including finite difference methods\cite{AlKeBo1974}, finite element methods\cite{RICHTER199465,Baoetal1998}, spectral methods\cite{KoBa1982,Shen2011}, and spectral element methods\cite{Komatitsch1999}. These methods are widely used in many applications. However, these methods have very strict requirements on the discretization --  in order to achieve good accuracy the mesh size has to be comparable to the wavelength or even smaller, which usually results in high memory and computational time cost and makes them difficult in the high-frequency simulation. The other category is asymptotic methods based on geometric optics\cite{EnRu:03,Runborg2007}. Compared to direct discretization to the wave equation, the asymptotic methods use special solution ansatz such as WKB and Gaussian beams so that they can resolve the small wavelength by multi-scale expansion with less restricted memory load and computational costs, and thus are more feasible for high-frequency wave propagation. However, the traditional WKB methods may produce unbounded solutions at caustics since the eikonal equation is of the Hamilton-Jacobi type. The Gaussian beam methods (GBM) can provide accurate solutions beyond caustics, but the convergence of GBM relies on a Taylor expansion determined by the width of beams, and when beams spread significantly, one loses accuracy in the wavefield. The frozen Gaussian approximation (FGA) uses a superposition of frozen (fixed-width) Gaussian functions to approximate the wavefield and thus can overcome the difficulty of beam spreading in GBM while maintaining the accuracy at caustics. The FGA was originally proposed in quantum chemistry for simulating Schr\"odinger equation by Heller in \cite{He:81} and Herman \& Kluk in \cite{HeKl:84}. Later on, a systematical justification of this method was developed in \cite{Ka:94,Ka:06,SwRo:09}. The FGA was generalized to high-frequency wave equations and strictly hyperbolic systems by the pioneer works \cite{LuYa:11,LuYa:CPAM,LuYa:MMS}, and recently to elastic wave equations \cite{hate2018fga}, semi-classical Dirac equations \cite{ChLo:19}, non-adiabatic dynamics for semi-classical
Schr\"odinger equations \cite{LuZh:2016FGAsh,LuZh:2018FGAsh,Huang2022}. The FGA has been applied to seismic tomography as an efficient forward model solver and has shown its advantage in both accuracy and parallelism \cite{chai2017frozen,Chai2021fgaconv}. 

We remark here that these implementations of the FGA mentioned above rely on the direct phase-space discretization of a two-fold integration form of the approximated solution (see equation \eqref{fga_ansatz} in the next section for details). This usually results in a large number of Gaussian functions, and in order to reconstruct the wavefield numerically, one has to sum them together, which makes the reconstruction the most computationally expensive step in the FGA algorithm. This disadvantage limits the development of the FGA in the areas where one needs to call the wavefield reconstruction iteratively over time propagation, such as in seismic tomography \cite{chai2018tomo}. To reduce the computational burden, stochastic approaches have been proposed for phase-space discretization. Kluk, Herman and David \cite{KlHeDa:86} employ a Monte Carlo algorithm to select the initial phase-space points in the study of quantum FGA propagation of a highly excitedly anharmonic oscillator. Later, the authors in \cite{LuZh:2016FGAsh,LuZh:2018FGAsh} give the FGA ansatz \eqref{fga_ansatz} a probabilistic interpretation and proposed a Monte-Carlo sampling algorithm to reconstruct physical observables for quantum surface hopping. Stochastic strategies have also been applied to phase-space approximations beyond FGA. For instance, Lasser and Sattlegger \cite{LaSa:17} applied a Monte Carlo method in the discretization of the Herman–Kluk propagator as a phase-space integral. Besides, a technique using the random batch method has been developed to reduce the computational cost in constructing the stochastic gradient for seismic tomography based on FGA \cite{Hu2021SeismicTW}. 

Recently, the authors in \cite{Xie2021} systematically studied the Monte Carlo approach to FGA's phase-space integral in the aspect of optimized sampling strategy and named it frozen Gaussian sampling (FGS). The main spirit of the FGS method is to view the FGA ansatz as an expectation of random variables on the phase-space and initialize the Gaussians randomly w.r.t. the underlying probability density function. The FGS can reduce the working load in reconstruction significantly and makes great success in simulating high-dimensional Schr\"odinger equations. In \cite{Xie2021}, sampling strategies are given for both Gaussian and WKB initial data and it is proved rigorously for Gaussian type initial data that the sampling error decays in the square-root order of the sampling number and does not depend on the asymptotic parameter.

Since the FGS has shown promising results for Schr\"odinger equations, it is natural to generalize this methodology to the FGAs for other equations, e.g., wave equations, Dirac equations, and hyperbolic systems. In this paper, we aim to develop efficient FGS algorithms to simulate high-frequency wave propagation and establish quantitative estimates for the sampling error. Two main differences arise in the scalar wave equation from the Schr\"odinger equation: first, the Hamiltonian of the wave equation takes the form of $H(\bm{q},\bm{p})=c(\bm{q})|\bm{p}|$ which has a singularity at $\bm{p}=\bm{0}$. This singularity introduces more difficulties in both designing the probability density function for the FGS and analyzing the sampling error. Second, the natural norm used for the Schr\"odinger equation is the $L^2$ norm, but for the wave equation, we need to deal with the high-frequency energy norm. 
{
We demonstrate that, with careful consideration of the probability formulation of the FGA, we can design efficient sampling algorithms for both Gaussian and WKB initial data. Rigorous error estimates can be established for the dimensionality $d\geq3$, and numerical results show the algorithms also perform well for $d=1,2$. Our analysis proves the sampling error does not depend on the asymptotic parameter for Gaussian initial data. Further, what remains unclear in \cite{Xie2021} is the rigorous error estimate for WKB initial data. In this paper, we use the method of stationary phase to analyze the error caused by approximating the probability density function and rigorously derive an error bound for the FGS of wave equations with WKB initial data.
}
 
The rest of this paper is organized as follows: in Section \ref{sec:FGS}, we briefly review the FGA for the scalar wave equation, then we introduce the FGS algorithm to compute the wave function and present a framework of the error analysis. In Section \ref{sec:ErrGauss} and \ref{sec:ErrWKB}, we focus on the wave equation with initial data in the form of Gaussian and WKB, respectively, and propose feasible sampling strategies to apply the FGS algorithm and the associated quantitative error analysis. In Section \ref{sec:num_ex}, we present a variety of numerical examples of applying the FGS algorithm to wave equations with both Gaussian and WKB initial data in dimensionality one to three. Section \ref{sec:Conclusion} concludes this paper.

\section{Frozen Gaussian Approximation and Frozen Gaussian Sampling}\label{sec:FGA-FGS}

In this section, we briefly review the frozen Gaussian approximation (FGA) to the scalar wave equations, which is based on an integral representation on the phase space with several variables evolved by an ODE system. Then we introduce the frozen Gaussian sampling (FGS) algorithm. The highlight of the FGS is that it gives the FGA ansatz a probabilistic interpretation and uses a Monte Carlo method to approximate the FGA ansatz. At the end of this section, we establish a framework of analyzing the sampling error due to the usage of stochastic simulation in the FGS.

\subsection{Frozen Gaussian Approximation for Scalar Wave Equations}\label{sec:FGA-SW}

We consider the wave equation \eqref{wave_eq} with initial conditions \eqref{init_con}. {In this subsection, we briefly outline the frozen Gaussian approximation of the solution of \eqref{wave_eq} without derivation. Interested readers may refer to \cite{chai2017frozen,LuYa:11} for more details.} The frozen Gaussian approximation (FGA) gives an integral representation of the solution wave function $u$ as follows, 
\begin{equation}\label{fga_ansatz}
\begin{aligned}
u_{\text{FGA}}^k (t,\bm{x}) & = \sum_{\pm}\int_{\mathbb{R}^d}\int_{\mathbb{R}^d}\frac{a_\pm(t,\bm{q},\bm{p})\psi_\pm^k(\bm{q},\bm{p})}{(2\pi/k)^{3d/2}}e^{\mathrm{i}k\Theta_\pm(t,\bm{x},\bm{q},\bm{p})}\mathrm{d}\bm{q}\mathrm{d}\bm{p} 
,
\end{aligned}
\end{equation}
where,
\begin{equation}\label{fga_ansatz_1}
\psi_{\pm}^k(\bm{q},\bm{p})=\int_{\mathbb{R}^d}u_{\pm,0}^k(\bm{y},\bm{q},\bm{p})e^{-\mathrm{i}k\bm{p}\cdot(\bm{y}-\bm{q})-\frac{k}{2}|\bm{y}-\bm{q}|^2}\mathrm{d}\bm{y},
\end{equation}
\begin{equation}\label{fga_ansatz_2}
u_{\pm,0}^k(\bm{y},\bm{q},\bm{p}) = \frac{1}{2}\left(f_0^k(\bm{y})\pm \frac{\i}{kc(\bm{q})|\bm{p}|} f_1^k(\bm{y})\right),
\end{equation}
The phase function $\Theta$ is given by
\begin{equation}\label{phase_func}
\Theta_{\pm}(t,\bm{x},\bm{q},\bm{p}) = \bm{P}_{\pm}(t,\bm{q},\bm{p})\cdot(\bm{x}-\bm{Q}_{\pm}(t,\bm{q},\bm{p}))+\frac{\mathrm{i}}{2}|\bm{x}-\bm{Q}_{\pm}(t,\bm{q},\bm{p})|^2.
\end{equation}
In equations (\ref{fga_ansatz}) to (\ref{phase_func}), $\i = \sqrt{-1}$ is the imaginary unit, and subscripts $+$ and $-$ indicate the two wave branches. The \emph{position center} $Q_{\pm}$, \emph{momentum center} $P_{\pm}$, and \emph{amplitude} $a_{\pm}$ are the time-dependent variables that arise from the FGA/FGS method, and they are computed by solving a system of ODEs which we outline below. 

The evolution of $\bm{Q}_{\pm}(t,\bm{q},\bm{p})$ and $\bm{P}_{\pm}(t,\bm{q},\bm{p})$ satisfies the ray tracing equations corresponding to the Hamiltonian $H_{\pm} = \pm c(\bm{Q}_{\pm}) |\bm{P}_{\pm}|$, that is, 
\begin{equation} \label{eq:ode_qp}
        \frac{\d \bm{Q}_{\pm}}{\d t} = \partial_{\bm{P}_{\pm}} H_{\pm}, \quad\text{and}\quad
        \frac{\d \bm{P}_{\pm}}{\d t} = \mp \partial_{\bm{Q}_{\pm}} H_{\pm}, 
\end{equation}
with initial conditions,
\begin{equation} \label{eq:init_qp}
    \bm{Q}_{\pm} (0,\bm{q},\bm{p}) = \bm{q},\quad\text{and}\quad \bm{P}_{\pm} (0,\bm{q},\bm{p}) = \bm{p}.
\end{equation}
The evolution of $a_{\pm}(t,\bm{q},\bm{p})$ is given by 
\begin{equation} \label{eq:ode_a}
    \frac{\d a_{\pm}}{\d t} = a_{\pm} \frac{\partial_{\bm{P}_{\pm}} H_{\pm} \cdot \partial_{\bm{Q}_{\pm}} H_{\pm}}{H_{\pm}} + \frac{a_{\pm}}{2} \mathrm{tr} \left(Z^{-1}_{\pm} \frac{\d Z_{\pm}}{\d t}\right),
\end{equation}
with the initial condition, 
\begin{equation} \label{eq:init_a}
    a_{\pm}(0,\bm{q},\bm{p}) = 2^{d/2}.
\end{equation}
Since the initial amplitude for each branch is equal to a constant $2^{d/2}$, we shall omit the subscripts $\pm$ in its notation and refer to it as $a(0,\bm{q},\bm{p})$ for brevity. 
In (\ref{eq:ode_a}), the following shorthand notations are used:
\begin{equation}
    \partial_{\bm{z}} = \partial_{\bm{q}} - \i \partial_{\bm{p}},\quad Z_{\pm} = \partial_{\bm{z}} \left(\bm{Q}_{\pm} + \i \bm{P}_{\pm}\right).
\end{equation}
Here $\partial_{\bm{z}} \bm{Q}_{\pm}$ and $\partial_{\bm{z}} \bm{P}_{\pm}$ are understood as matrices, with the $(j,k)$ component of $\partial_{\bm{z}} \bm{Q}_{\pm}$ given by $\partial_{\bm{z}_j} {\bm{Q}_{\pm}}_k$. 

\subsection{Frozen Gaussian Sampling for Scalar Wave Equations}\label{sec:FGS}

In Section \ref{sec:FGA-SW}, we point out that the FGA ansatz (\ref{fga_ansatz}) is an integral on the phase space. If we can write the FGA ansatz in the form of expectations, then we may approximate the integral through a Monte Carlo method. This is the idea of the \textbf{frozen Gaussian sampling} (FGS). In this subsection, we discuss how to implement the FGS for general scalar wave equations of dimensionality $d$. 

First notice that the FGA ansatz (\ref{fga_ansatz}) is a summation of two integrals that represent two wave branches respectively, that is,
\begin{equation}
    u_{\mathrm{FGA}}^k (t,\bm{x}) = u_{\mathrm{FGA},+}^k (t,\bm{x}) + u_{\mathrm{FGA},-}^k (t,\bm{x}).
\end{equation}
Here we use the notations,
\begin{align}
u_{\mathrm{FGA},\pm}^k (t,\bm{x}) = \frac{1}{(2\pi/k)^{3d/2}}\int_{\mathbb{R}^{2d}}a_{\pm}(t,z_0)\psi_{\pm}^k(z_0)e^{\mathrm{i}k\Theta_{\pm}(t,\bm{x},z_0)}\mathrm{d}z_0.
\end{align}

With this form, we may write the FGA ansatz into a summation of two expectations. We introduce two real-valued probability density functions $\pi_+ (\cdot)$ and $\pi_- (\cdot)$ on the phase space $\mathbb{R}^{2d}$, each of which then defines a probability measure $\mathbb{P}_{\pm}$ on $\mathbb{R}^{2d}$:
\begin{align}
    \mathbb{P}_{\pm}(\Omega) = \int_{\Omega} \pi_{\pm} (z_0) \mathrm{d} z_0, \text{ for all } \Omega \subset \mathbb{R}^{2d},
\end{align}
where $\mathbb{P}_{\pm}(\mathbb{R}^{2d}) = \int_{\mathbb{R}^{2d}} \pi_{\pm} (z_0) \mathrm{d} z_0 = 1$. 

With the density functions $\pi_{\pm}$, we can rewrite the FGA ansatz (\ref{fga_ansatz}) as 
\begin{align}
u_{\mathrm{FGA}}^k (t,\bm{x}) &= \sum_{\pm}\frac{1}{(2\pi/k)^{3d/2}}\int_{\mathbb{R}^{2d}}\pi_\pm(z_0)\frac{a_\pm(t,z_0)\psi_\pm^k(z_0)}{\pi_\pm(z_0)}e^{\mathrm{i}k\Theta_\pm(t,\bm{x},z_0)}\mathrm{d}z_0 \nonumber \\
&= \sum_{\pm}\mathbb{E}_{z_0\sim \pi_\pm} \left[\frac{1}{(2\pi/k)^{3d/2}}\frac{a_\pm(t,z_0)\psi_\pm^k(z_0)}{\pi_\pm(z_0)}e^{\mathrm{i}k\Theta_\pm(t,\bm{x},z_0)} \right] 
\label{fga_expectation},
\end{align}
where $z_0 = (\bm{q}, \bm{p})$. 

Then we may use a Monte Carlo method for $u_{\mathrm{FGA}}^k$. We can obtain $M^+$ samples from the density $\pi_+$ and $M^-$ samples from the density $\pi_-$, but for simplicity we set $M^+ = M^- = M$. From each density function, we generate $M$ random samples, denoted by $\{z_{0,+}^{(j)}\}_{j=1}^M$ from the density $\pi_+$ and $\{z_{0,-}^{(j)}\}_{j=1}^M$ from the density $\pi_-$, and approximate $u_{\mathrm{FGA}}^k$ through the following:
\begin{align} \label{eq:fgs_wavefunc}
u_{\mathrm{FGA}}^k (t,\bm{x}) & \approx u_{\mathrm{FGS}}^k\left(t,\bm{x}\right) 
= \frac{1}{M}\sum_{j=1}^M \Lambda_+\left(t,\bm{x},z_{0,+}^{(j)}\right) + \frac{1}{M}\sum_{j=1}^M \Lambda_-\left(t,\bm{x},z_{0,-}^{(j)}\right),
\end{align}
where
\begin{equation}
\Lambda_{\pm}\left(t,\bm{x},z_{0,\pm}^{(j)}\right) = \frac{1}{(2\pi/k)^{3d/2}}\frac{a_{\pm}\left(t,z_{0,\pm}^{(j)}\right)\psi_{\pm}^k\left(z_{0,\pm}^{(j)}\right)}{\pi_{\pm}\left(z_{0,\pm}^{(j)}\right)}\exp\left(\mathrm{i}k\Theta_{\pm}\left(t,\bm{x},z_{0,\pm}^{(j)}\right)\right),
\end{equation}
and $u^k_{\mathrm{FGS}}$ is referred to as the frozen Gaussian sampling (FGS) wave function. Later we will show that in Gaussian or WKB initial data cases, the probability density functions $\pi_+$ and $\pi_-$ can be selected as the same. In this situation, we denote $\pi = \pi_{\pm}$ for simplicity, and a set of samples $\{z_0^{(j)}\}_{j=1}^M$ can be obtained from a single probability density function $\pi$. Note that $\Lambda_{\pm}(t,\bm{x},z_{0})$ is treated as a random variable, which is a function of $z_{0}$, with the expectation
\begin{equation}
\mathbb{E} \left[\Lambda_{\pm}(t,\bm{x},z_0)\right] = u_{\mathrm{FGA},\pm}^k (t,\bm{x}).
\end{equation}
Here we denote that
\begin{align}
u_{\mathrm{FGS},\pm}^k\left(t,\bm{x};M,\left\{z_0^{(j)}\right\}_{j=1}^M\right) = \frac{1}{M}\sum_{j=1}^M \Lambda_{\pm}\left(t,\bm{x},z_0^{(j)}\right).
\end{align}

Algorithmically, one may sample $z_0$ according to the probability density functions $\pi_{\pm}$ respectively to obtain two sets of random samples $\{z_{0,+}^{(j)}\}_{j=1}^M$ and $\{z_{0,-}^{(j)}\}_{j=1}^M$, and then obtain the FGA variables through evolving the ODE system (\ref{eq:ode_qp}) and (\ref{eq:ode_a}) up to time $t$. With enough samples, one can approximate the FGS wave function numerically by (\ref{eq:fgs_wavefunc}). Summarizing above, we give the pseudo-code of frozen Gaussian sampling (FGS) in Algorithm \ref{ag:fgs_ag}. 

\begin{algorithm}[htb]
    \caption{Frozen Gaussian Sampling} \label{ag:fgs_ag}
    \begin{algorithmic}[1]
        \State \textbf{Initial sampling:} generate two sets of initial samples $\{z_{0,+}^{(j)}\}_{j=1}^M$ and $\{z_{0,-}^{(j)}\}_{j=1}^M$ from the probability density functions $\pi_+$ and $\pi_-$ respectively on the phase space $\mathbb{R}^{2d}$.  
        \State \textbf{Time evolution:} for each branch and each $j = 1:M$, evolve the ODE system (\ref{eq:ode_qp}) and (\ref{eq:ode_a}) up to time $t$ with initial conditions (\ref{eq:init_qp}) and (\ref{eq:init_a}) to obtain the FGA variables $\bm{Q}_{\pm}(t,z_0^{(j)})$, $\bm{P}_{\pm}(t,z_0^{(j)})$ and $a_{\pm}(t,z_0^{(j)})$, where $z_0^{(j)}$ stands for both $z_{0,+}^{(j)}$ and $z_{0,-}^{(j)}$.
        \State \textbf{Reconstruction:} compute the FGS wave function based on (\ref{eq:fgs_wavefunc}). 
    \end{algorithmic}
\end{algorithm}

\subsection{A Framework of Error Analysis}\label{sec:FGS-Err}

In this subsection, we introduce a framework of error analysis for applying the FGS algorithm to wave equations. 

We define the high-frequency energy norm of $u$ as follows: 
\begin{equation}
    \Vert u \Vert_E = \frac{1}{k}\left(\Vert\partial_t u \Vert_{L^2} + \Vert \nabla_{\bm{x}} u \Vert_{L^2} \right).
\end{equation}
We use this energy norm defined above to measure the approximation error of the FGS algorithm, that is, 
\begin{align}
    E_0 := \Vert u_{\mathrm{FGS}} - u \Vert_E \leq \Vert u_{\mathrm{FGA}} - u \Vert_E + \Vert u_{\mathrm{FGS}} - u_{\mathrm{FGA}} \Vert_E := E_{\mathrm{FGA}} + E_S. 
\end{align}
Here $E_{\mathrm{FGA}} = \Vert u_{\mathrm{FGA}} - u \Vert_E$ is the asymptotic error of the FGA ansatz and $E_S = \Vert u_{\mathrm{FGS}} - u_{\mathrm{FGA}} \Vert_E$ is the Monte Carlo sampling error. It has been proved  in \cite{LuYa:CPAM} that with proper smoothness assumptions on the sound velocity $c(\bm{x})$ and high-frequency initial conditions (\ref{init_con}), we may arrive at 
\begin{align} \label{eq:error_fga}
    \Vert u_{\mathrm{FGA}} - u \Vert_E \leq C_A(t)k^{-1},
\end{align}
for a given time $t > 0$, where $C_A(t)$ is a constant depending on $t$ but not on $k$. Therefore, our main concern is to analyze the sampling error. 

We define the mean square of sampling error $\mathcal{E}_S$ and of total error $\mathcal{E}_0$ as follows:
\begin{align}
    & \mathcal{E}_S = \mathbb{E} \left\Vert u_{\mathrm{FGA}}(t,\cdot) - u_{\mathrm{FGS}}\left(t,\cdot;M,\left\{ z_0^{(j)} \right\}_{j=1}^M \right) \right\Vert_{E}^2, \quad{\text{and }} \mathcal{E}_0 = \mathbb{E} \left\Vert u(t,\cdot) - u_{\mathrm{FGS}}\left(t,\cdot;M,\left\{ z_0^{(j)} \right\}_{j=1}^M \right) \right\Vert_{E}^2. \nonumber
\end{align}

The following lemma gives a method for estimating the sampling error of applying the FGS algorithm to wave equations regardless of the specific form of the initial conditions and the chosen sampling density functions. 

\begin{lemma} \label{lm:frame_error_anaz}
    Consider the scalar wave equation (\ref{wave_eq}) up to time $t$ with initial conditions (\ref{init_con}) that belong to the Schwartz class. When applying the frozen Gaussian sampling (FGS) with sampling density functions $\pi_{\pm}$, 
    \begin{itemize}
        \item[(1)] the quadratic expectation of the sampling error $E_S$, that is, $\mathcal{E}_S$ follows
        \begin{align}\label{ieq:lm2-1}
            \mathcal{E}_S \leq & \frac{4}{M}\frac{1}{k^2} \left( \int_{\mathbb{R}^d} \mathbb{E} \left|\partial_t \Lambda_\pm (t,\bm{x},z_0)\right|^2 \mathrm{d}\bm{x} +  \int_{\mathbb{R}^d} \mathbb{E} \left|\nabla_{\bm{x}} \Lambda_\pm (t,\bm{x},z_0)\right|^2 \mathrm{d}\bm{x}\right) ,
        \end{align}
        \item[(2)] there exists a positive constant $C_0(t,d)$ depending on time $t$ and dimensionality $d$ but $k$ such that 
        \begin{align} \label{ieq:lm_2}
            & \frac{1}{k^2} \int_{\mathbb{R}^d} \mathbb{E} \left|\partial_t \Lambda_{\pm} (t,\bm{x},z_0)\right|^2 \mathrm{d}\bm{x} + \frac{1}{k^2} \int_{\mathbb{R}^d} \mathbb{E} \left|\nabla_{\bm{x}} \Lambda_{\pm} (t,\bm{x},z_0)\right|^2 \mathrm{d}\bm{x} 
            \leq \frac{C_0(t,d)}{2^{3d} (\pi/k)^{\frac{5d}{2}}} \int_{\mathbb{R}^{2d}} \frac{\left|a_{\pm}(t,z_0)\psi_{\pm}^k(z_0)\right|^2}{\pi_{\pm}(z_0)} \mathrm{d}z_0.
        \end{align}
    \end{itemize}
    whenever $\mathbb{E} \left|\partial_t \Lambda_{\pm} (t,\bm{x},z_0)\right|^2 < \infty$ and $\mathbb{E} \left|\nabla_{\bm{x}} \Lambda_{\pm} (t,\bm{x},z_0)\right|^2 < \infty$. 
\end{lemma}

\begin{proof}
    (1) Notice that
    \begin{align}
        u_{\mathrm{FGA}}^k\left(t,\bm{x}\right)  = u_{\mathrm{FGA},+}^k\left(t,\bm{x}\right) + u_{\mathrm{FGA},-}^k\left(t,\bm{x}\right), 
        \quad\text{and }
        u_{\mathrm{FGS}}^k\left(t,\bm{x}\right)  = u_{\mathrm{FGS},+}^k\left(t,\bm{x}\right) + u_{\mathrm{FGS},-}^k\left(t,\bm{x}\right). \nonumber
    \end{align}
Using the triangle inequality of energy norm gives 
    \begin{align}
        \left\Vert u_{\mathrm{FGS}}^k\left(t,\cdot \right) - u_{\mathrm{FGA}}^k\left(t,\cdot \right) \right\Vert_E & \leq \left\Vert u_{\mathrm{FGS},+}^k\left(t,\cdot \right) - u_{\mathrm{FGA},+}^k\left(t,\cdot \right) \right\Vert_E + \left\Vert u_{\mathrm{FGS},-}^k\left(t,\cdot \right) - u_{\mathrm{FGA},-}^k\left(t,\cdot \right) \right\Vert_E.
    \end{align}
Then
\begin{align}
    \left\Vert u_{\mathrm{FGS}}^k\left(t,\cdot \right) - u_{\mathrm{FGA}}^k\left(t,\cdot \right) \right\Vert_E^2 \leq 2 \left( \left\Vert u_{\mathrm{FGS},+}^k\left(t,\cdot \right) - u_{\mathrm{FGA},+}^k\left(t,\cdot \right) \right\Vert_E^2 + 
    \left\Vert u_{\mathrm{FGS},-}^k\left(t,\cdot \right) - u_{\mathrm{FGA},-}^k\left(t,\cdot \right) \right\Vert_E^2\right).
\end{align}
Taking expectations on both sides, we see
\begin{align} \label{quad_expect}
    \mathbb{E}\left\Vert u_{\mathrm{FGS}}^k\left(t,\cdot \right) - u_{\mathrm{FGA}}^k\left(t,\cdot \right) \right\Vert_E^2 \leq 2 \left( \mathbb{E} \left\Vert u_{\mathrm{FGS},+}^k\left(t,\cdot \right) - u_{\mathrm{FGA},+}^k\left(t,\cdot \right) \right\Vert_E^2  + \right.\nonumber \\
    \left. 
    \mathbb{E} \left\Vert u_{\mathrm{FGS},-}^k\left(t,\cdot \right) - u_{\mathrm{FGA},-}^k\left(t,\cdot \right) \right\Vert_E^2 \right). 
\end{align}
The estimation for $\mathbb{E} \left\Vert u_{\mathrm{FGS},+}^k\left(t,\cdot \right) - u_{\mathrm{FGA},+}^k\left(t,\cdot \right) \right\Vert_E^2$ and $\mathbb{E} \left\Vert u_{\mathrm{FGS},-}^k\left(t,\cdot \right) - u_{\mathrm{FGA},-}^k\left(t,\cdot \right) \right\Vert_E^2$ are exactly the same. For simplicity, we will denote $u_{\mathrm{FGA},\pm}^k$ as $u_{\mathrm{FGA}}$, $u_{\mathrm{FGS},\pm}^k$ as $u_{\mathrm{FGS}}$ and omit the $\pm$ subscripts for other variables in the following derivation.

By the definition of the energy norm,
\begin{align}
   \left\Vert u_{\mathrm{FGS}}\left(t,\cdot \right) - u_{\mathrm{FGA}}\left(t,\cdot \right) \right\Vert_E^2 = \frac{1}{k^2} \big( & \left\Vert\partial_t \left(u_{\mathrm{FGS}}\left(t,\cdot \right) - u_{\mathrm{FGA}}\left(t,\cdot \right)\right)\right\Vert_{L^2} + 
   \left\Vert\nabla_{\bm{x}} \left(u_{\mathrm{FGS}}\left(t,\cdot \right) - u_{\mathrm{FGA}}\left(t,\cdot \right) \right)\right\Vert_{L^2} \big)^2.
\end{align}
Then
\begin{align}
    \left\Vert u_{\mathrm{FGS}}\left(t,\cdot \right) - u_{\mathrm{FGA}}\left(t,\cdot \right) \right\Vert_E^2 \leq \frac{2}{k^2} \left(\left\Vert\partial_t \left(u_{\mathrm{FGS}}\left(t,\cdot \right) - u_{\mathrm{FGA}}\left(t,\cdot \right)\right)\right\Vert_{L^2}^2 + 
    \left\Vert\nabla_{\bm{x}} \left(u_{\mathrm{FGS}}\left(t,\cdot \right) - u_{\mathrm{FGA}}\left(t,\cdot \right) \right)\right\Vert_{L^2}^2 \right).
\end{align}
Taking expectations on both sides of the above inequality, we have
\begin{align}
    \mathbb{E}\left\Vert u_{\mathrm{FGS}}\left(t,\cdot \right) - u_{\mathrm{FGA}}\left(t,\cdot \right) \right\Vert_E^2 \leq \frac{2}{k^2} \left( \mathbb{E} \left\Vert\partial_t \left(u_{\mathrm{FGS}}\left(t,\cdot \right) - u_{\mathrm{FGA}}\left(t,\cdot \right)\right)\right\Vert_{L^2}^2 + \right. \nonumber \\
    \left. \mathbb{E} \left\Vert\nabla_{\bm{x}} \left(u_{\mathrm{FGS}}\left(t,\cdot \right) - u_{\mathrm{FGA}}\left(t,\cdot \right) \right)\right\Vert_{L^2}^2 \right).
\end{align}
Simplify the right-hand side of the above inequality as follows:
\begin{align}
    \mathbb{E} \left\Vert\partial_t \left(u_{\mathrm{FGS}}\left(t,\cdot \right) - u_{\mathrm{FGA}}\left(t,\cdot \right)\right)\right\Vert_{L^2}^2 
    = & \int_{\mathbb{R}^d} \mathbb{E} \left|\partial_t u_{\mathrm{FGS}}(t,\bm{x}) - \partial_t u_{\mathrm{FGA}}(t,\bm{x}) \right|^2 \mathrm{d}\bm{x} \nonumber \\
    = & \int_{\mathbb{R}^d} \mathbb{E} \left| \frac{1}{M} \sum_{j=1}^M \left(\partial_t \Lambda \left(t,\bm{x},z_0^{(j)}\right) - \partial_t u_{\mathrm{FGA}}\left(t,\bm{x}\right) \right) \right|^2 \mathrm{d} \bm{x} \nonumber \\
    = & \frac{1}{M} \int_{\mathbb{R}^d} \mathbb{E} \left| \partial_t \Lambda \left(t,\bm{x},z_0\right) - \partial_t u_{\mathrm{FGA}}\left(t,\bm{x}\right) \right|^2 \mathrm{d} \bm{x} \nonumber \\
    \leq & \frac{1}{M} \int_{\mathbb{R}^d} \mathbb{E} \left| \partial_t \Lambda \left(t,\bm{x},z_0\right) \right|^2 \mathrm{d} \bm{x}.
\end{align}
The last two steps of the above derivation follow from the fact that $z_0^{(j)}$ are independent identically distributed and 
\begin{align}
    \mathbb{E}\left[\partial_t\Lambda(t,\bm{x},z_0)\right] & = \partial_t u_{\mathrm{FGA}}(t,\bm{x}),
\end{align}
that is, the expectation of $\partial_t \Lambda (t, \bm{x}, z_0)$ regarding the random variable $z_0$ is $\partial_t u_{\mathrm{FGA}} (t,\bm{x})$. In other words, $\mathbb{E} \left| \partial_t \Lambda \left(t,\bm{x},z_0\right) - \partial_t u_{\mathrm{FGA}}\left(t,\bm{x}\right) \right|^2$ is the variance of $\partial_t \Lambda (t, \bm{x}, z_0)$, and so we have 
\begin{align}
    \mathbb{E} \left[\left| \partial_t \Lambda \left(t,\bm{x},z_0\right) - \partial_t u_{\mathrm{FGA}}\left(t,\bm{x}\right) \right|^2\right] & = \mathbb{E} \left[\partial_t \Lambda \left(t, \bm{x}, z_0\right)^2\right] - \left(\mathbb{E}\left[\partial_t\Lambda\left(t,\bm{x},z_0\right)\right]\right)^2 
    \leq \mathbb{E} \left[\partial_t \Lambda \left(t, \bm{x}, z_0\right)^2\right].
\end{align}

Similarly,
\begin{align}
    \mathbb{E} \left\Vert\nabla_{\bm{x}} \left(u_{\mathrm{FGS}}\left(t,\cdot \right) - u_{\mathrm{FGA}}\left(t,\cdot \right)\right)\right\Vert_{L^2}^2 
    = & \int_{\mathbb{R}^d} \mathbb{E} \left|\nabla_{\bm{x}} u_{\mathrm{FGS}}(t,\bm{x}) - \nabla_{\bm{x}} u_{\mathrm{FGA}}(t,\bm{x}) \right|^2 \mathrm{d}\bm{x} \nonumber \\
    \leq & \frac{1}{M} \int_{\mathbb{R}^d} \mathbb{E} \left| \nabla_{\bm{x}} \Lambda \left(t,\bm{x},z_0\right) \right|^2 \mathrm{d} \bm{x}.
\end{align}
Therefore, 
\begin{align} \label{quad_expect_sig}
    \mathbb{E}\left\Vert u_{\mathrm{FGS}}\left(t,\cdot \right) - u_{\mathrm{FGA}}\left(t,\cdot \right) \right\Vert_E^2 \leq \frac{2}{M}\frac{1}{k^2}  \left( \int_{\mathbb{R}^d} \mathbb{E} \left| \partial_t \Lambda \left(t,\bm{x},z_0\right) \right|^2 \mathrm{d} \bm{x} + \right. 
    \left. \int_{\mathbb{R}^d} \mathbb{E} \left| \nabla_{\bm{x}} \Lambda \left(t,\bm{x},z_0\right) \right|^2 \mathrm{d} \bm{x} \right) .
\end{align}
Then
\begin{align}
    \mathbb{E}\left\Vert u_{\mathrm{FGS},\pm}^k\left(t,\cdot \right) - u_{\mathrm{FGA},\pm}^k\left(t,\cdot \right) \right\Vert_E^2 \leq \frac{2}{M} \frac{1}{k^2} \left( \int_{\mathbb{R}^d} \mathbb{E} \left| \partial_t \Lambda_{\pm} \left(t,\bm{x},z_0\right) \right|^2 \mathrm{d} \bm{x} + \right. 
    \left. \int_{\mathbb{R}^d} \mathbb{E} \left| \nabla_{\bm{x}} \Lambda_{\pm} \left(t,\bm{x},z_0\right) \right|^2 \mathrm{d} \bm{x} \right) .
\end{align}
and \eqref{ieq:lm2-1} follows.

(2) In the following derivation, we also omit the $\pm$ subscripts in variables. 

First estimate the upper bound of $\frac{1}{k^2} \int_{\mathbb{R}^d} \mathbb{E} \left| \partial_t \Lambda \left(t,\bm{x},z_0\right) \right|^2 \mathrm{d} \bm{x}$. The random variable is 
\begin{align}
    \Lambda(t,\bm{x},z_0) = \frac{1}{(2\pi/k)^{\frac{3d}{2}}} \frac{a(t,z_0)\psi(z_0)}{\pi(z_0)} \exp\left(\mathrm{i}k\bm{P}(t,z_0)\cdot(\bm{x}- \bm{Q}(t,z_0)) - \frac{k}{2} |\bm{x}-\bm{Q}(t,z_0)|^2 \right) .
\end{align}
Taking the derivative of $\Lambda$ with respect to $t$ gives
\begin{align}
    \partial_t \Lambda & (t,\bm{x},z_0) = \frac{1}{(2\pi/k)^{\frac{3d}{2}}} \frac{1}{\pi(z_0)} \exp\left(\mathrm{i}k\bm{P}(t,z_0)\cdot(\bm{x}- \bm{Q}(t,z_0)) - \frac{k}{2} |\bm{x}-\bm{Q}(t,z_0)|^2 \right) \times \nonumber \\
    & \left\{ \partial_t [a(t,z_0)\psi(z_0)] + a(t,z_0)\psi(z_0) \left(\mathrm{i}k \partial_t [\bm{P}(t,z_0)\cdot(\bm{x}-\bm{Q}(t,z_0))] - \frac{k}{2} \partial_t \left[ |\bm{x} - \bm{Q}(t,z_0)|^2 \right] \right) \right\} .
\end{align}
Then 
\begin{multline}
    |\partial_t \Lambda (t,\bm{x},z_0) |^2 = \frac{1}{(2\pi/k)^{3d}} \frac{|a(t,z_0)\psi(z_0)|^2}{|\pi(z_0)|^2} \exp\left(-k|\bm{x}- \bm{Q}(t,z_0)|^2\right) 
    \\\times
    \left| \frac{\partial_t a(t,z_0)}{a(t,z_0)} + \left(\mathrm{i}k \partial_t [\bm{P}(t,z_0)\cdot(\bm{x}-\bm{Q}(t,z_0))] - \frac{k}{2} \partial_t \left[ |\bm{x} - \bm{Q}(t,z_0)|^2 \right] \right) \right|^2 .
\end{multline}
Taking the expectation of $|\partial_t \Lambda (t,\bm{x},z_0) |^2$, we obtain
\begin{multline}
    \frac{1}{k^2} \mathbb{E} \left| \partial_t \Lambda \left(t,\bm{x},z_0\right) \right|^2 = \frac{1}{(2\pi/k)^{3d}} \int_{\mathbb{R}^{2d}} \pi(z_0) \frac{|a(t,z_0)\psi(z_0)|^2}{|\pi(z_0)|^2} \exp\left(-k|\bm{x}- \bm{Q}(t,z_0)|^2\right) \\ \times 
    \left| \frac{1}{k} \frac{\partial_t a(t,z_0)}{a(t,z_0)}  + \left(\mathrm{i} \partial_t [\bm{P}(t,z_0)\cdot(\bm{x}-\bm{Q}(t,z_0))] - \frac{1}{2} \partial_t |\bm{x} - \bm{Q}(t,z_0)|^2 \right) \right|^2 \mathrm{d}z_0 .
\end{multline}
Denote that 
\begin{align}
    \lambda_1(t,\bm{x},z_0) := \left| \frac{1}{k} \frac{\partial_t a(t,z_0)}{a(t,z_0)}  + \left(\mathrm{i} \partial_t [\bm{P}(t,z_0)\cdot(\bm{x}-\bm{Q}(t,z_0))] - \frac{1}{2} \partial_t |\bm{x} - \bm{Q}(t,z_0)|^2 \right) \right|^2.
\end{align}
Then
\begin{align}
    \frac{1}{k^2} \int_{\mathbb{R}^d} \mathbb{E} \left| \partial_t \Lambda \left(t,\bm{x},z_0\right) \right|^2 \mathrm{d} \bm{x} 
    = & \frac{1}{(2\pi/k)^{3d}} \int_{\mathbb{R}^{2d}} \frac{|a(t,z_0)\psi(z_0)|^2}{\pi(z_0)} \left(\int_{\mathbb{R}^d} \lambda_1(t,\bm{x},z_0) \exp\left(-k|\bm{x}- \bm{Q}(t,z_0)|^2\right) \mathrm{d}\bm{x}\right) \mathrm{d}z_0 .
\end{align}
To compute the integral 
$\displaystyle
    \int_{\mathbb{R}^d} \lambda_1(t,\bm{x},z_0) \exp\left(-k|\bm{x}- \bm{Q}(t,z_0)|^2\right) \mathrm{d} \bm{x},
$ 
  one may rewrite $\lambda_1$ as 
\begin{align}
    \lambda_1 & = \bigg|\frac{1}{k} \mathrm{Re}\left(\frac{\partial_t a(t,z_0)}{a(t,z_0)}\right) - \frac{1}{2} \partial_t\left[|\bm{x} - \bm{Q}(t,z_0)|^2\right] + 
    \mathrm{i} \left( \partial_t \left[\bm{P}(t,z_0)\cdot(\bm{x} - \bm{Q}(t,z_0))\right] + \frac{1}{k} \mathrm{Im}\left(\frac{\partial_t a(t,z_0)}{a(t,z_0)}\right)\right)\bigg|^2 \nonumber \\
    & = \left(\frac{1}{k} \mathrm{Re}\left(\frac{\partial_t a(t,z_0)}{a(t,z_0)}\right) - \frac{1}{2}\partial_t\left[|\bm{x} - \bm{Q}(t,z_0)|^2\right]\right)^2 + 
    \left(\partial_t \left[\bm{P}(t,z_0)\cdot(\bm{x} - \bm{Q}(t,z_0))\right] + \frac{1}{k} \mathrm{Im}\left(\frac{\partial_t a(t,z_0)}{a(t,z_0)}\right)\right)^2 \nonumber \\
    & = \frac{1}{k^2}\left|\frac{\partial_t a(t,z_0)}{a(t,z_0)}\right|^2 - \frac{1}{k} \mathrm{Re}\left(\frac{\partial_t a(t,z_0)}{a(t,z_0)}\right) \partial_t\left[|\bm{x} - \bm{Q}(t,z_0)|^2\right] \nonumber \\
    & \hspace{3em} + \frac{1}{4} \left(\partial_t \left[|\bm{x} - \bm{Q}(t,z_0)|^2\right]\right)^2 + \left(\partial_t \left[\bm{P}(t,z_0)\cdot(\bm{x} - \bm{Q}(t,z_0))\right]\right)^2 \nonumber \\
    & \hspace{3em} + \frac{2}{k} \partial_t \left[\bm{P}(t,z_0)\cdot(\bm{x} - \bm{Q}(t,z_0))\right] \mathrm{Im}\left(\frac{\partial_t a(t,z_0)}{a(t,z_0)}\right).
\end{align}
Based on the fact that
\begin{equation} \label{eq:exp_fact}
    \begin{split}
        & \int_{\mathbb{R}^d} \exp\left(-k|\bm{x}- \bm{Q}(t,z_0)|^2\right) \mathrm{d} \bm{x} = \pi^{\frac{d}{2}} k^{-\frac{d}{2}}, \\ 
        & \int_{\mathbb{R}^d} \left(\bm{x} - \bm{Q}(t,z_0)\right) \exp\left(-k|\bm{x}- \bm{Q}(t,z_0)|^2\right) \mathrm{d} \bm{x} = 0, \\ 
        & \int_{\mathbb{R}^d} \left|\bm{x} - \bm{Q}(t,z_0)\right|^2 \exp\left(-k|\bm{x}- \bm{Q}(t,z_0)|^2\right) \mathrm{d} \bm{x} = \frac{d}{2} \pi^{\frac{d}{2}} k^{-\frac{d}{2}-1}, 
    \end{split} 
\end{equation}
we find after some simple calculations that
\begin{align}
    \int_{\mathbb{R}^d} \lambda_1(t,\bm{x},z_0) \exp\left(-k|\bm{x}- \bm{Q}(t,z_0)|^2\right) \mathrm{d} \bm{x} \leq \sigma_1(t,z_0) \pi^{\frac{d}{2}} k^{-\frac{d}{2}} + \sigma_2(t,z_0) \pi^{\frac{d}{2}} k^{-\frac{d}{2}-1} + \sigma_3(t,z_0) \pi^{\frac{d}{2}} k^{-\frac{d}{2}-2},
\end{align}
where
\begin{align}
    & \sigma_1(t,z_0) = |\bm{P}(t,z_0)\cdot \partial_t \bm{Q}(t,z_0) |^2, \nonumber \\
    & \sigma_2(t,z_0) = \frac{d}{2} |\partial_t \bm{P}(t,z_0) |^2 + \frac{d}{2} |\partial_t \bm{Q}(t,z_0) |^2 -2 \mathrm{Im} \left(\frac{\partial_t a(t,z_0)}{ a(t,z_0)}\right) \bm{P}(t,z_0)\cdot \partial_t \bm{Q}(t,z_0), \nonumber \\
    & \sigma_3(t,z_0) = \left|\frac{\partial_t a(t,z_0)}{a(t,z_0)}\right|^2. \nonumber
\end{align}
Therefore,
\begin{align}
    \int_{\mathbb{R}^d} \lambda_1(t,\bm{x},z_0) \exp\left(-k|\bm{x}- \bm{Q}(t,z_0)|^2\right) \mathrm{d} \bm{x} \leq (\sigma_1(t,z_0) + \sigma_2(t,z_0) + \sigma_3(t,z_0)) \pi^{\frac{d}{2}} k^{-\frac{d}{2}}.
\end{align} 
Thus
\begin{align} \label{integral_e_1}
    & \frac{1}{k^2} \int_{\mathbb{R}^d} \mathbb{E} \left| \partial_t \Lambda \left(t,\bm{x},z_0\right) \right|^2 \mathrm{d} \bm{x} 
    \leq \frac{1}{2^{3d}(\pi/k)^{\frac{5d}{2}}} \int_{\mathbb{R}^{2d}} \frac{|a(t,z_0)\psi(z_0)|^2}{\pi(z_0)} (\sigma_1(t,z_0) + \sigma_2(t,z_0) + \sigma_3(t,z_0)) \mathrm{d} z_0.
\end{align}

Next estimate the upper bound of $\frac{1}{k^2} \int_{\mathbb{R}^d} \mathbb{E} \left| \nabla_{\bm{x}} \Lambda \left(t,\bm{x},z_0\right) \right|^2 \mathrm{d} \bm{x}$. Taking the derivative of $\Lambda$ with respect to $\bm{x}$ gives
\begin{multline}
    \nabla_{\bm{x}} \Lambda (t,\bm{x},z_0) =  \frac{1}{(2\pi/k)^{\frac{3d}{2}}} \frac{a(t,z_0)\psi(z_0)}{\pi(z_0)}  \exp\left(\mathrm{i}k\bm{P}(t,z_0)\cdot(\bm{x}- \bm{Q}(t,z_0)) - \frac{k}{2} |\bm{x}-\bm{Q}(t,z_0)|^2 \right) 
    \\\times 
    [ \mathrm{i} k \bm{P}(t,z_0) - k (\bm{x} - \bm{Q}(t,z_0))].
\end{multline}
Then
\begin{align}
    |\nabla_{\bm{x}} \Lambda (t,\bm{x},z_0)|^2 = & \frac{1}{(2\pi/k)^{3d}} \frac{|a(t,z_0)\psi(z_0)|^2}{|\pi(z_0)|^2}  \exp\left( - k |\bm{x}-\bm{Q}(t,z_0)|^2 \right) 
    k^2 \left(|\bm{P}(t,z_0)|^2 + |\bm{x} - \bm{Q}(t,z_0)|^2 \right).
\end{align}
Taking the expectation of $ |\nabla_{\bm{x}} \Lambda (t,\bm{x},z_0)|^2$, we have
\begin{multline}
    \frac{1}{k^2} \mathbb{E} \left| \nabla_{\bm{x}} \Lambda \left(t,\bm{x},z_0\right) \right|^2    = \frac{1}{(2\pi/k)^{3d}} \int_{\mathbb{R}^{2d}} \pi (z_0) \frac{|a(t,z_0)\psi(z_0)|^2}{|\pi(z_0)|^2}  \exp\left( - k |\bm{x}-\bm{Q}(t,z_0)|^2 \right) 
    \\\times\left(|\bm{P}(t,z_0)|^2 + |\bm{x} - \bm{Q}(t,z_0)|^2 \right) \mathrm{d} z_0.
\end{multline}
Denote that 
$\displaystyle
    \lambda_2(t,\bm{x},z_0) := |\bm{P}(t,z_0)|^2 + |\bm{x} - \bm{Q}(t,z_0)|^2.
$
 Then
\begin{align}
    \frac{1}{k^2} \int_{\mathbb{R}^d} \mathbb{E} \left| \nabla_{\bm{x}} \Lambda \left(t,\bm{x},z_0\right) \right|^2 \mathrm{d}\bm{x} 
    = & \frac{1}{(2\pi/k)^{3d}} \int_{\mathbb{R}^{2d}} \frac{|a(t,z_0)\psi(z_0)|^2}{\pi(z_0)} \left(\int_{\mathbb{R}^d} \lambda_2(t,\bm{x},z_0) \exp\left(-k|\bm{x}- \bm{Q}(t,z_0)|^2\right) \mathrm{d}\bm{x}\right) \mathrm{d}z_0.
\end{align}
Again use the fact (\ref{eq:exp_fact}), we obtain
\begin{align}
    \int_{\mathbb{R}^d} \lambda_2(t,\bm{x},z_0) \exp\left(-k|\bm{x}- \bm{Q}(t,z_0)|^2\right) \mathrm{d} \bm{x} = |\bm{P}(t,z_0)|^2 \pi^{\frac{d}{2}} k^{-\frac{d}{2}} + \frac{d}{2} \pi^{\frac{d}{2}} k^{- \frac{d}{2}-1}.
\end{align}
Let 
$\displaystyle
    \sigma_4(t,z_0) = |\bm{P}(t,z_0)|^2 + \frac{d}{2}. 
$
 Then
\begin{align}
    \int_{\mathbb{R}^d} \lambda_2(t,\bm{x},z_0) \exp\left(-k|\bm{x}- \bm{Q}(t,z_0)|^2\right) \mathrm{d} \bm{x} \leq \sigma_4(t,z_0) \pi^{\frac{d}{2}} k^{-\frac{d}{2}}.
\end{align}
Thus 
\begin{align} \label{integral_e_2}
    \frac{1}{k^2} \int_{\mathbb{R}^d} \mathbb{E} \left| \nabla_{\bm{x}} \Lambda \left(t,\bm{x},z_0\right) \right|^2 \mathrm{d} \bm{x} \leq \frac{1 }{2^{3d}(\pi/k)^{\frac{5d}{2}}} \int_{\mathbb{R}^{2d}} \frac{|a(t,z_0)\psi(z_0)|^2}{\pi(z_0)} \sigma_4(t,z_0) \mathrm{d} z_0.
\end{align}

Combining (\ref{integral_e_1}) and (\ref{integral_e_2}), we arrive at
\begin{align}
    & \frac{1}{k^2} \int_{\mathbb{R}^d} \mathbb{E} \left| \partial_t \Lambda \left(t,\bm{x},z_0\right) \right|^2 \mathrm{d} \bm{x} + \frac{1}{k^2} \int_{\mathbb{R}^d} \mathbb{E} \left| \nabla_{\bm{x}} \Lambda \left(t,\bm{x},z_0\right) \right|^2 \mathrm{d} \bm{x} 
    \leq \frac{1}{2^{3d} (\pi/k)^{\frac{5d}{2}}} \int_{\mathbb{R}^{2d}} \frac{|a(t,z_0)\psi(z_0)|^2}{\pi(z_0)}  \sigma(t,z_0) \mathrm{d} z_0,
\end{align}
where $\sigma = \sigma_1 + \sigma_2 + \sigma_3 + \sigma_4$. Since the initial wavefield (\ref{init_con}) belongs to the Schwartz class, it decreases rapidly as $|\bm{q}|+|\bm{p}| \rightarrow \infty$ in the initial decomposition of FGA algorithm. According to Proposition 3.4 of \cite{LuYa:CPAM}, the FGA variables $(\bm{Q}, \bm{P})$ and their derivatives with respect to $t$ are bounded for all $z_0 = (\bm{q},\bm{p})$. Further, according to Lemma 5.1 of \cite{LuYa:CPAM}, ${\mathrm{det}\left(Z\right)}^{-1}$ is bounded. Considering the proportional relationship of $a$ and $\mathrm{det}(Z)$ mentioned in \cite{YaLuFo:13}:
\begin{align} \label{eq:a_z_prop}
    \frac{a^2}{\mathrm{det}\left(Z\right)} = constant,
\end{align}
it follows that the amplitude $a$ and its derivative with respect to $t$ are also bounded for all $z_0$. Therefore, there exists a positive constant $C_0(t,d)$ depending on time $t > 0$ and dimensionality $d$ that for all $z_0$,
\begin{align}
    \sigma(t,z_0) = & | \bm{P}(t,z_0)\cdot \partial_t \bm{Q}(t,z_0) |^2 +  \frac{d}{2} |\partial_t \bm{P}(t,z_0) |^2 + \frac{d}{2} |\partial_t \bm{Q}(t,z_0) |^2  + \left|\frac{\partial_t a(t,z_0)}{a(t,z_0)}\right|^2  \nonumber \\
     & - 2 \mathrm{Im} \left(\frac{\partial_t a(t,z_0) }{ a(t,z_0) }\right) \bm{P}(t,z_0)\cdot \partial_t \bm{Q}(t,z_0) + |\bm{P}(t,z_0)|^2 + \frac{d}{2} \leq C_0(t,d) .
\end{align}
Note that $C_0(t,d)$ is independent of $k$. Thus (\ref{ieq:lm_2}) follows. 
\end{proof}

Lemma \ref{lm:frame_error_anaz} states that, in order to decide a bound for the mean square of the sampling error $\mathcal{E}_S$, we may focus on estimating the integrals $\int_{\mathbb{R}^{2d}} \left|a_{\pm}(t,z_0)\psi_{\pm}^k(z_0)\right|^2 / \pi_{\pm}(z_0) \mathrm{d}z_0$. To do so, we need to decide the sampling density functions $\pi_{\pm}$. The choice for the sampling density functions is related to the initial conditions. In the following cases, we consider Gaussian initial conditions and WKB initial conditions and present how to design suitable sampling functions. Then we use Lemma \ref{lm:frame_error_anaz} to analyze the sampling error in both cases. Notice that in the above lemma, we assume the integrability of $\left|a_{\pm}(t,z_0)\psi_{\pm}^k(z_0)\right|^2 / \pi_{\pm}(z_0)$ with respect to $z_0$ on $\mathbb{R}^{2d}$. In Theorem \ref{thm:gauss_init} and \ref{thm:wkb_init}, we prove the integrability in $d \geq 3$, and in Section \ref{sec:num_ex}, our numerical examples also indicate the integrability in $d = 1$ and $2$. 

The following lemma is also useful in the analysis of the next two sections.

\begin{lemma} \label{lm:a_bounded}
    Consider the FGA variable $a(t,\bm{q},\bm{p})$. For given time $t > 0$, there exists a positive constant $K(t,d)$ depending on $t$ but $k$ such that for any $(\bm{q},\bm{p}) \in \mathbb{R}^{2d}$, 
    \begin{align}
        \left|\frac{a(t,\bm{q},\bm{p})}{a(0,\bm{q},\bm{p})}\right| \leq K(t,d).
    \end{align}
\end{lemma}

\begin{proof}
    As mentioned in the proof of Lemma \ref{lm:frame_error_anaz}, the amplitude $a$ has a proportional relationship with ${\mathrm{det}\left(Z\right)}^{-1}$  as given in (\ref{eq:a_z_prop}) and thus is bounded for any $(\bm{q},\bm{p})$ given time $t > 0$ (Recall that ${\mathrm{det}\left(Z\right)}^{-1}$  is bounded according to Lemma 5.1 of \cite{LuYa:CPAM}). Note that $a(0,\bm{q},\bm{p}) = 2^{\frac{d}{2}}$ is a constant which does not depend on $(\bm{q},\bm{p})$. Thus the conclusion is obvious. 
\end{proof}

To conclude this section, we briefly summarize the form of initial conditions that will be discussed in the following two sections. In this paper, we mainly consider wave equation (\ref{wave_eq}) with initial conditions in the form of
\begin{equation} \label{init_gen}
    \begin{cases}
        u(0,\bm{x}) = h_1(\bm{x}), \\
        \partial_t u(0,\bm{x}) = k h_2(\bm{x}),
    \end{cases}
\end{equation}
where $h_1(\bm{x})$ and $h_2(\bm{x})$ are either both normalized Gaussian functions or both normalized WKB functions \footnote{The Wentzel–Kramers–Brillouin (WKB) functions arise from the WKB approximation where the wave function is assumed to be an exponential function with amplitude and phase that change slowly. Such a wave function is called a WKB function. In the context of approximating the wave equation, the amplitude and phase of the WKB function are indifferent to the frequency parameter $k$ (See Section \ref{sec:ErrWKB}). }. One may decompose (\ref{wave_eq}) with initial conditions (\ref{init_gen}) into two equations with initial conditions being
\begin{equation} \label{init_gen_1}
    \begin{cases}
        u_1 (0,\bm{x}) = h_1(\bm{x}), \\
        \partial_t u_1(0,\bm{x}) = 0,
    \end{cases}
\end{equation}
and
\begin{equation} \label{init_gen_2}
    \begin{cases}
        u_2 (0,\bm{x}) = 0, \\
        \partial_t u_2(0,\bm{x}) = k h_2(\bm{x}),
    \end{cases}
\end{equation}
respectively. Then the solution of (\ref{wave_eq}) is $u = u_1 + u_2$. Note that in (\ref{init_gen_1}) the $\partial_t u(0,\bm{x})$ term vanishes, thus solving wave equation with initial conditions (\ref{init_gen_1}) by the FGS is very similar to solving Schrödinger equation with corresponding initial conditions, which does not include any initial condition on the $\partial_t u$ (See \cite{Xie2021}). Thus we focus on solving (\ref{wave_eq}) with (\ref{init_gen_2}) here. To be consistent with the notations in Section \ref{sec:FGA-SW}, $f_0^k(\bm{x})$ and $f_1^k(\bm{x})$ are used to denote the initial conditions $u(0,\bm{x})$ and $\partial_t u(0,\bm{x})$ respectively in the following context. 

\section{Frozen Gaussian Sampling with Gaussian Initial Conditions}\label{sec:ErrGauss}

In this section, we discuss the FGS for wave equation (\ref{wave_eq}) with Gaussian initial conditions and analyze the sampling and total error. Without loss of generality, we consider the following Gaussian initial data:
\begin{equation} \label{eq:init_gauss}
    \begin{cases}
        f_0^k(\bm{x})  = 0, \\
        f_1^k(\bm{x})  = k \left(\prod_{j=1}^d a_j \right)^{\frac{1}{4}} (\pi/k)^{-\frac{d}{4}} \exp\left(\mathrm{i}k\bm{\tilde{p}} \cdot (\bm{x} - \bm{\tilde{q}}) \right) \exp\left( - \sum_{j=1}^d \frac{a_j}{2/k} (x_j - \tilde{q_j})^2 \right),
    \end{cases}
\end{equation}
where $\bm{\tilde{q}} \in \mathbb{R}^d$, $\bm{0} \ne \bm{\tilde{p}} \in \mathbb{R}^d$ and $a_j (1\leq j \leq d)$ are positive. 

Plug the initial condition (\ref{eq:init_gauss}) into the expression of $\psi^k_{\pm}$ given in (\ref{fga_ansatz_1}) and (\ref{fga_ansatz_2}):
\begin{align} \label{eq:psi_gauss}
\psi_{\pm}^k(\bm{q},\bm{p})  =~& \int_{\mathbb{R}^d}   \pm \frac{1}{2} \frac{\mathrm{i}}{c(\bm{q})|\bm{p}|}   f_1^k(\bm{y}) e^{-\mathrm{i}k\bm{p}\cdot(\bm{y}-\bm{q})-\frac{k}{2}|\bm{y}-\bm{q}|^2} \mathrm{d} \bm{y} \nonumber \\
 =~&  \pm \frac{1}{2} \frac{\mathrm{i}}{c(\bm{q})|\bm{p}|}  2^{\frac{d}{2}} \pi^{\frac{d}{4}} k^{-\frac{d}{4}} \prod_{j=1}^d \left(\frac{\sqrt{a_j}}{1+a_j}\right)^{\frac{1}{2}} \exp\left(-\frac{(\tilde{p_j}-p_j)^2 + a_j(\tilde{q_j}-q_j)^2}{2(1+a_j)/k}\right) \nonumber \\
& ~\times\exp\left(\frac{\mathrm{i}k(a_j\tilde{q_j}+q_j)(\tilde{p_j} - p_j)}{1+a_j} + \mathrm{i}k(p_jq_j - \tilde{p_j}\tilde{q_j}) \right).
\end{align}
and define
\begin{align}
    \tilde{\psi}^k(\bm{q},\bm{p})  =~& \int_{\mathbb{R}^d} f_1^k(\bm{y}) e^{-\mathrm{i}k\bm{p}\cdot(\bm{y}-\bm{q})-\frac{k}{2}|\bm{y}-\bm{q}|^2} \mathrm{d} \bm{y} \nonumber \\
     =~& 2^{\frac{d}{2}} \pi^{\frac{d}{4}} k^{-\frac{d}{4}} \prod_{j=1}^d \left(\frac{\sqrt{a_j}}{1+a_j}\right)^{\frac{1}{2}} \exp\left(-\frac{(\tilde{p_j}-p_j)^2 + a_j(\tilde{q_j}-q_j)^2}{2(1+a_j)/k}
    \right) \nonumber \\
    & ~\times\exp\left(
    \frac{\mathrm{i}k(a_j\tilde{q_j}+q_j)(\tilde{p_j} - p_j)}{1+a_j} + \mathrm{i}k(p_jq_j - \tilde{p_j}\tilde{q_j}) \right).
\end{align}
Considering the initial amplitude $a(0,\bm{q},\bm{p})$ and the modulus of $a(0,z_0)\tilde{\psi}^k(z_0)$, we choose a probability density function as follows:
\begin{equation} \label{choice_for_pdf}
\pi(z_0) = \frac{1}{\mathcal{Z}} |a(0,z_0)\tilde{\psi}^k(z_0)|,
\end{equation}
and use this density function to generate initial random samples for each wave branch, that is, $\pi_{\pm} = \pi$. Here $\mathcal{Z}$ is a normalization parameter to ensure that $\pi(z_0)$ is a probability density function such that 
    $\displaystyle\int_{\mathbb{R}^{2d}} \pi(z_0) \mathrm{d}z_0 = 1$. 
The integral of $|a(0,z_0)\tilde{\psi}^k(z_0)|$ on the phase space is 
\begin{align} \label{int_phase_space_2}
    \int_{\mathbb{R}^{2d}}|a(0,z)\tilde{\psi}^k(z)|\mathrm{d}z = 2^{2d} \pi^{\frac{5d}{4}} k^{-\frac{5d}{4}} \prod_{j=1}^d \left(\frac{1+a_j}{\sqrt{a_j}}\right)^{\frac{1}{2}}.
\end{align}
Therefore,
\begin{align} \label{eq:normal_gauss}
    \mathcal{Z} = 2^{2d} \pi^{\frac{5d}{4}} k^{-\frac{5d}{4}} \prod_{j=1}^d \left(\frac{1+a_j}{\sqrt{a_j}}\right)^{\frac{1}{2}}.
\end{align}
By choosing probability density function as (\ref{choice_for_pdf}), we actually obtain a multivariate normal distribution:
\begin{align} \label{eq:pdf_gauss}
    \pi(\bm{q},\bm{p}) = (2\pi)^{-d} k^d \prod_{j=1}^d \left( \frac{\sqrt{a_j}}{1+a_j} \right) \exp \left(-\sum_{j=1}^d \frac{(\tilde{p_j}-p_j)^2 + a_j(\tilde{q_j}-q_j)^2}{2(1+a_j)/k} \right).
\end{align}
This density function is easy to sample. Moreover, we will show later in the error analysis that this sampling method also gives satisfactory results about the sampling error.

We now present a rigorous error estimate for the FGS in Gaussian initial data cases in the following theorem. The proof of the theorem is applicable to wave equations with dimensionality $d \geq 3$. In Section \ref{sec:num_ex}, we will manifest through several numerical examples that there are similar results for wave equations with dimensionality $d = 1$ or $d = 2$.

\begin{theorem} \label{thm:gauss_init}
    Consider the $d \geq 3$ dimensional wave equation (\ref{wave_eq}) up to time $t$ with the normalized Gaussian initial conditions (\ref{eq:init_gauss}). Suppose that the velocity function $c(\bm{x})$ is bounded in $\mathbb{R}^d$ with its infimum $c_{\mathrm{inf}} = \inf_{\bm{x}\in \mathbb{R}^d} c(\bm{x}) > 0$. When applying the frozen Gaussian sampling (FGS) with sampling density functions $\pi_{\pm}$ as in (\ref{eq:pdf_gauss}), there exists a positive constant $C(t,d)$ which is independent of $k$ such that 
    \begin{equation} \label{eq:gauss_es}
        \mathcal{E}_S \leq \frac{C(t,d)}{M}, 
    \end{equation}
    and
    \begin{equation} \label{eq:gauss_e0}
        \mathcal{E}_0 \leq (C_A(t)k^{-1})^2 + \frac{C(t,d)}{M} .
    \end{equation}
\end{theorem}

\begin{proof} Omit the $\pm$ subscript for ease of notation. According to Lemma (\ref{lm:frame_error_anaz}) (2) and Lemma (\ref{lm:a_bounded}), 
\begin{align} \label{eq:bg_gauss}
    & \frac{1}{k^2} \int_{\mathbb{R}^d} \mathbb{E} \left| \partial_t \Lambda \left(t,\bm{x},z_0\right) \right|^2 \mathrm{d} \bm{x} + \frac{1}{k^2} \int_{\mathbb{R}^d} \mathbb{E} \left| \nabla_{\bm{x}} \Lambda \left(t,\bm{x},z_0\right) \right|^2 \mathrm{d} \bm{x} \nonumber \\
    \leq & C_0(t,d) \frac{\mathcal{Z}}{2^{3d}(\pi/k)^{\frac{5d}{2}}} \int_{\mathbb{R}^{2d}} \frac{|a(t,z_0)\psi(z_0)|^2}{|a(0,z_0)\tilde{\psi}(z_0)|} \mathrm{d} z_0 
    \nonumber \\
    \leq & C_0(t,d) (K(t,d))^2 \frac{\mathcal{Z}}{2^{3d}(\pi/k)^{\frac{5d}{2}}} \int_{\mathbb{R}^{2d}} \frac{1}{2} \frac{1}{c(\bm{q}) |\bm{p}| } |a(0,z_0)\psi(z_0)| \mathrm{d} z_0.
\end{align}
Now we estimate integral of $ \frac{1}{2} \frac{1}{c(\bm{q}) |\bm{p}| } |a(0,z_0)\psi(z_0)| $ on the phase space. Note that  
\begin{align}
    |a(0,\bm{q},\bm{p})\psi(\bm{q},\bm{p})| & = \frac{1}{2} \frac{1}{c(\bm{q})|\bm{p}|} 2^d \pi^{\frac{d}{4}} k^{-\frac{d}{4}} \prod_{j=1}^d \left(\frac{\sqrt{a_j}}{1+a_j}\right)^{\frac{1}{2}} 
    \exp\left(-\sum_{j=1}^d \frac{(\tilde{p_j}-p_j)^2 + a_j(\tilde{q_j}-q_j)^2}{2(1+a_j)/k}\right).
\end{align}
Since $c(\bm{x}) \geq c_{\mathrm{inf}} > 0$, we have
\begin{align} \label{eq:iqip}
    & I = \int_{\mathbb{R}^{2d}} \frac{1}{2} \frac{1}{c(\bm{q}) |\bm{p}|} |a(0,\bm{q},\bm{p})\psi(\bm{q},\bm{p})| \mathrm{d} \bm{q} \mathrm{d} \bm{p} \nonumber \\
    \leq & \int_{\mathbb{R}^{2d}}\frac{1}{4} \frac{1}{c_{\mathrm{inf}}^2|\bm{p}|^2} 2^d \pi^{\frac{d}{4}} k^{-\frac{d}{4}} \prod_{j=1}^d \left(\frac{\sqrt{a_j}}{1+a_j}\right)^{\frac{1}{2}} \exp\left(-\sum_{j=1}^d \frac{(\tilde{p_j}-p_j)^2 + a_j(\tilde{q_j}-q_j)^2}{2(1+a_j)/k}\right) \mathrm{d} \bm{q} \mathrm{d} \bm{p} \nonumber \\
    = & 2^{d-2} \pi^{\frac{d}{4}} k^{-\frac{d}{4}} \prod_{j=1}^d \left(\frac{\sqrt{a_j}}{1+a_j}\right)^{\frac{1}{2}} \int_{\mathbb{R}^{2d}} \frac{1}{c_{\mathrm{inf}}^2|\bm{p}|^2} \exp\left(-\sum_{j=1}^d \frac{(\tilde{p_j}-p_j)^2 + a_j(\tilde{q_j}-q_j)^2}{2(1+a_j)/k}\right) \mathrm{d} \bm{q} \mathrm{d} \bm{p} \nonumber \\
    =&  2^{d-2} \pi^{\frac{d}{4}} k^{-\frac{d}{4}} \prod_{j=1}^d \left(\frac{\sqrt{a_j}}{1+a_j}\right)^{\frac{1}{2}} \int_{\mathbb{R}^d} \exp\left(-\sum_{j=1}^d \frac{ a_j(\tilde{q_j}-q_j)^2}{2(1+a_j)/k}\right) \mathrm{d} \bm{q} \nonumber \\
    & \hspace{7em} \times \int_{\mathbb{R}^d}  \frac{1}{c_{\mathrm{inf}}^2|\bm{p}|^2} \exp\left(-\sum_{j=1}^d \frac{(\tilde{p_j}-p_j)^2}{2(1+a_j)/k}\right) \mathrm{d} \bm{p} =: k^{-\frac{d}{4}} C_1 I_q I_p,
\end{align}
where we denote
    $C_1 = 2^{d-2} \pi^{\frac{d}{4}} \prod_{j=1}^d \left(\frac{\sqrt{a_j}}{1+a_j}\right)^{\frac{1}{2}}$ 
and
\begin{align}
    I_p = \int_{\mathbb{R}^d}  \frac{1}{c_{\mathrm{inf}}^2|\bm{p}|^2} \exp\left(-\sum_{j=1}^d \frac{(\tilde{p_j}-p_j)^2}{2(1+a_j)/k}\right) \mathrm{d} \bm{p}, \quad
    I_q = \int_{\mathbb{R}^d} \exp\left(-\sum_{j=1}^d \frac{ a_j(\tilde{q_j}-q_j)^2}{2(1+a_j)/k}\right) \mathrm{d} \bm{q}.
\end{align}
Take $h = \left( \sum_{j=1}^d \frac{\tilde{p}_j^2}{1+a_j}\right)^{\frac{1}{2}} > 0$. Split $I_p$ into two parts:
\begin{align}
    I_p =  \left( \int_{B(0,\frac{h}{2})} + \int_{\mathbb{R}^d\backslash B(0,\frac{h}{2})} \right) \frac{1}{c_{\mathrm{inf}}^2|\bm{p}|^2} \exp\left(-\sum_{j=1}^d \frac{(\tilde{p_j}-p_j)^2}{2(1+a_j)/k}\right) \mathrm{d} \bm{p} =: I_{p,1} + I_{p,2}.
\end{align}
For the first part of $I_p$, 
We apply the spherical coordinates transform in the estimation of $I_{p,1}$ and obtain
\begin{align}
    I_{p,1} & \leq \int_{B(0,\frac{h}{2})} \frac{1}{c_{\mathrm{inf}}^2|\bm{p}|^2} \exp \left(-\frac{k}{2} \frac{h^2}{4}\right) \mathrm{d} \bm{p} = c_{\mathrm{inf}}^{-2} \exp\left(-\frac{h^2}{8}k\right) \times \nonumber \\
    & \int_0^{\frac{h}{2}} \int_0^\pi \cdots \int_0^\pi \int_0^{2\pi} r^{d-3} \sin^{d-2}\phi_1  \cdots \sin^{2} \phi_{d-3} \sin \phi_{d-2} \d r \d \phi_1 \cdots \d \phi_{d-2} \d \phi_{d-1}. \nonumber
\end{align}
When $d \geq 3$, the integral on the right-hand side of the above inequality is convergent and its value is independent of $k$. Besides, given any $d > 0$, there is a bound $C_d$ such that $\exp(-x) \leq C_d x^{-\frac{d}{2}}$ for all $x > 0$. Therefore, we obtain 
\begin{align}
    I_{p,1}  \leq \tilde{C}_d c_{\mathrm{inf}}^{-2} \left(\frac{h^2}{8} k\right)^{-\frac{d}{2}} = \tilde{C}_d c_{\mathrm{inf}}^{-2} 2^{\frac{3d}{2}} h^{-d} k^{-\frac{d}{2}},
\end{align}
where $\tilde{C}_d$ is a constant depending on $d$. For the second part of $I_p$,
\begin{align}
    I_{p,2} & \leq \int_{\mathbb{R}^d} \frac{1}{c_{\mathrm{inf}}^2} \frac{4}{h^2} \exp\left(-\sum_{j=1}^d \frac{(\tilde{p_j}-p_j)^2}{2(1+a_j)/k}\right) \mathrm{d} \bm{p} 
    = c_{\mathrm{inf}}^{-2} 2^{\frac{d}{2}+2} \pi^{\frac{d}{2}} h^{-2} k^{-\frac{d}{2}} \prod_{j=1}^d (1+a_j)^{\frac{1}{2}}.
\end{align}
Therefore,
\begin{align}
    I_p =I_{p,1} + I_{p,2} \leq c_{\mathrm{inf}}^{-2} \left(2^{\frac{3d}{2}} h^{-d} \tilde{C}_d + 2^{\frac{d}{2}+2} \pi^{\frac{d}{2}} h^{-2} \prod_{j=1}^d (1+a_j)^{\frac{1}{2}} \right) k^{-\frac{d}{2}} .
\end{align}
Computing $I_q$ gets
\begin{align}
    \int_{\mathbb{R}^d} \exp\left(-\sum_{j=1}^d \frac{ a_j(\tilde{q_j}-q_j)^2}{2(1+a_j)/k}\right) \mathrm{d} \bm{q} = 2^{\frac{d}{2}} \pi^{\frac{d}{2}} k^{-\frac{d}{2}} \prod_{j=1}^d \left(\frac{1+a_j}{a_j}\right)^{\frac{1}{2}}.
\end{align}
Therefore,
\begin{align} 
    I_q I_p \leq \frac{1}{c_{\mathrm{inf}}^2} \left(2^{2d} \pi^{\frac{d}{2}} h^{-d} \tilde{C}_d \prod_{j=1}^d \left(\frac{1+a_j}{a_j}\right)^{\frac{1}{2}} + 2^{d+2} \pi^d h^{-2} \prod_{j=1}^d \left(\frac{1+a_j}{\sqrt{a_j}}\right) \right) k^{-d} =: C_2 k^{-d}.
\end{align}
Then by (\ref{eq:iqip}) one has
\begin{align} \label{eq:int_phase_space_1}
    \int_{\mathbb{R}^d} \int_{\mathbb{R}^d} \frac{1}{2} \frac{1}{c (\bm{q}) |\bm{p}|} |a(0,\bm{q},\bm{p})\psi(\bm{q},\bm{p})| \mathrm{d} \bm{q} \mathrm{d} \bm{p} \leq \tilde{C}_0 k^{-{\frac{5d}{4}}},
\end{align}
where $\tilde{C}_0 = C_1C_2$ are constants which are independent of $k$.

Therefore, together with (\ref{eq:int_phase_space_1}) and (\ref{eq:normal_gauss}), (\ref{eq:bg_gauss}) implies
\begin{align}
    \frac{1}{k^2} \int_{\mathbb{R}^d} \mathbb{E} \left| \partial_t \Lambda \left(t,\bm{x},z_0\right) \right|^2 \mathrm{d} \bm{x} + \frac{1}{k^2} \int_{\mathbb{R}^d} \mathbb{E} \left| \nabla_{\bm{x}} \Lambda \left(t,\bm{x},z_0\right) \right|^2 \mathrm{d} \bm{x} \leq \tilde{C}(t,d).
\end{align}
where we denote 
\begin{align}\label{eq:gauss_tld_c}
    \tilde{C}(t,d) & = C_0(t,d) (K(t,d))^2 \tilde{C}_0 2^{-d} \pi^{-\frac{5d}{4}} \prod_{j=1}^d \left(\frac{1+a_j}{\sqrt{a_j}}\right)^{\frac{1}{2}} \nonumber \\
    & = C_0(t,d) (K(t,d))^2 c_{\mathrm{inf}}^{-2} \left( 2^{2d-2} \pi^{-\frac{d}{2}} h^{-d} \tilde{C}_d \prod_{j=1}^d \left( \frac{1+a_j}{a_j} \right)^{\frac{1}{2}} + 2^d h^{-2} \prod_{j=1}^d \left( \frac{1+a_j}{\sqrt{a_j}} \right) \right).
\end{align}
According to Lemma (\ref{lm:frame_error_anaz}) (1), we obtain
\begin{align}
    \mathbb{E}\left\Vert u_{\mathrm{FGS}}^k\left(t,\cdot \right) - u_{\mathrm{FGA}}^k\left(t,\cdot \right) \right\Vert_E^2 \leq \frac{8 \tilde{C}(t,d)}{M}.
\end{align}
This proves (\ref{eq:gauss_es}), and (\ref{eq:gauss_e0}) follows directly from (\ref{eq:error_fga}). 
\end{proof}

\begin{remark}
    We discuss how $C(t,d)$ in (\ref{eq:gauss_es}) and (\ref{eq:gauss_e0}) depends on the dimensionality $d \geq 3$. Generally, $C(t,d)$ increases exponentially as $d$ increases. The dependence of $C(t,d)$ on $d$ is very similar to the corresponding constant in the error estimate of applying the FGS to Sch\"odinger equations with Gaussian initial data, and we refer to \cite{Xie2021} for a detailed analysis. 
\end{remark}

\begin{remark} \label{rmk:adv_gauss}
    We remark on the advantages of the FGS over other existing methods such as the finite difference method. For simplicity, we ignore the asymptotic error $E_{\mathrm{FGA}} = \Vert u_{\mathrm{FGS}} - u_{\mathrm{FGA}} \Vert_E $ of the FGA ansatz. Then the approximation error $E_0$ of the FGS algorithm can be estimated by the sampling error $E_S$, that is,
    \begin{equation}
        E_0 = \Vert u - u_{\mathrm{FGS}} \Vert_E \approx \left(\mathcal{E}_S\right)^{\frac{1}{2}} \leq \sqrt{\frac{C(t,d)}{M}}. 
    \end{equation}
    In order to guarantee that $E_0 \leq \delta$, the sample size of the FGS algorithm should be more than 
    \begin{equation}
        \mathcal{N}_{\mathrm{FGS}} = \frac{C(t,d)}{\delta^2}, \quad \delta > 0.
    \end{equation}
    The least sample size is independent of the wave number $k$. On the contrary, if we use other mesh-based numerical methods such as the finite difference method, the spatial step length of the mesh should be at least of order $\mathcal{O}(k^{-1})$ to capture the high-frequency wavefield \cite{AlKeBo1974}. For wave equations with dimensionality $d$, the size of the spatial degrees of freedom should be at least 
    \begin{equation} \label{eq:fd_order}
        \mathcal{N}_{\mathrm{FD}} = \mathcal{O}(k^d).
    \end{equation}
    However, the sample size $\mathcal{N}_{\mathrm{FGS}}$ is a $d$-th power of an $\mathcal{O}(1)$ parameter. Therefore, when the wave number $k \gg 1$, there is a great saving on computational costs to solve the wave equation with the FGS compared to other mesh-based methods.  
\end{remark}

\section{Frozen Gaussian Sampling with WKB Initial Conditions}\label{sec:ErrWKB}

In this section, we discuss the FGS for wave equation (\ref{wave_eq}) with WKB initial conditions and analyze the approximation error. Generally, we consider the following WKB initial data:
\begin{equation} \label{eq:init_wkb}
    \begin{cases}
        f_0^k(\bm{x}) = 0, \\
        f_1^k(\bm{x}) = k a_{\mathrm{in}}(\bm{x})\exp\left(\i k S_{\mathrm{in}}(\bm{x})\right),
    \end{cases}
\end{equation}
where $a_{\text{in}}$ and $S_{\mathrm{in}}$ are two real-valued functions on $\mathbb{R}^d$, which are both independent of $k$. The initial amplitude $a_{\mathrm{in}}$ follows 
\begin{align}
    \int_{\mathbb{R}^d} a_{\mathrm{in}}^2(\bm{x}) \mathrm{d} \bm{x} = 1.
\end{align}

In order to determine a suitable sampling density function, we give the following assumptions. 

\begin{assumption} \label{ass:inverse}
    For any $\bm{p}\in \mathbb{R}^d$, there exists at most one $\bm{y} \in \mathbb{R}^d$ s.t. $\nabla S_{\mathrm{in}}(\bm{y}) = \bm{p}$, i.e., $\nabla S_{\mathrm{in}}$ is an injective mapping. The inverse of $\nabla S_{\mathrm{in}}$ is given by $T: \bm{p} \mapsto \bm{y}$. The domain of $T$ is denoted by $\mathcal{D}(T)$, i.e.,
    \begin{align}
        T:\mathcal{D}(T) \rightarrow \mathbb{R}^d, \quad \bm{p} \mapsto \bm{y} \nonumber.
    \end{align} 
\end{assumption}

\begin{assumption} \label{ass:inverse_smooth}
    The inverse mapping T of $\nabla S_{\mathrm{in}}$ is smooth. This implies that $T$ is continuous. Therefore, assume that there exists a positive constant $M_T$ that does not depend on $k$ such that for any $\bm{p} \in \mathcal{D}(T)$,
    \begin{align}
        |T(\bm{p})| \leq M_T|\bm{p}| . \nonumber
    \end{align}
    Let $\bm{y} = T(\bm{p})$, that is, $\bm{p} = \nabla S_{\mathrm{in}}(\bm{y})$, the above inequality is equivalent to
    \begin{align}
        |\bm{y}| \leq M_T|\nabla S_{\mathrm{in}}(\bm{y})| .\nonumber
    \end{align}
\end{assumption}

\begin{assumption} \label{ass:jacobian_bounded}
    For any $\bm{y} \in \mathbb{R}^d$, 
    \begin{align}
        0 < \left| \mathrm{det}\left(\nabla^2 S_{\mathrm{in}}(\bm{y})\right)\right| \leq C^* .\nonumber
    \end{align}
    where $C^*$ is a positive constant that is independent of $k$.
\end{assumption}

The stationary phase method is used to acquire a suitable probability function as well. 
\begin{lemma} \label{lm:stat_phase}
    Let $a(\bm{x})$ and $\Psi(\bm{x})$ be two smooth functions on $\mathbb{R}^d$. Suppose that $a(\bm{x}) \in C^{\infty}_0(\mathbb{R}^d)$ and $\Psi(\bm{x}) \in C^{\infty}(\mathbb{R}^d)$. Suppose that $\nabla \Psi(\bm{x})$ has a finite number of zero points $\{\bm{y}_l\}_{l=1}^N$ on the support of $a$ and $\det\left(\nabla^2 \Psi(\bm{y}_l)\right) \ne 0$ for each zero points $\bm{y}_l$. Consider the highly oscillatory integration:
    \begin{align}
        I_{k} = \int_{\mathbb{R}^d} a(\bm{x}) \exp\left(\i k\Psi(\bm{x})\right) \mathrm{d}\bm{x}.
    \end{align}
    As $k \rightarrow \infty$, $I_{k}$ has the following asymptotic behavior:
    \begin{align}
        I_{k} = (2\pi)^{\frac{d}{2}} k^{-\frac{d}{2}} \sum_{l=1}^{N} \frac{\exp\left(\i k\Psi(\bm{y}_l) + \frac{\i \pi}{4} \mathrm{sgn}\left(\nabla^2 \Psi(\bm{y}_l)\right)\right)}{|\det\nabla^2\Psi(\bm{y}_l)|^{\frac{1}{2}}} \left(a(\bm{y}_l) + \sum_{n=1}^\infty k^{-n} \bm{c}_n \cdot D^{2n} a(\bm{y}_l)\right),
    \end{align}
    where $\mathrm{sgn}\left(\nabla^2\Psi(\bm{y}_l)\right)$ denotes the number of positive eigenvalues of the matrix $\nabla^2 \Psi(\bm{y}_l)$ minus the number of negative eigenvalues of $\nabla^2 \Psi(\bm{y}_l)$, and $\bm{c}_n \in \mathbb{R}^{d^{2n}} (n=1,2,\dots)$ are constant vectors that do not depend on $k$. 
\end{lemma}

A proof of Lemma \ref{lm:stat_phase} can be found in \cite{Evans:98}.

Plug the initial condition (\ref{eq:init_wkb}) into $\psi^k_{\pm}$ given in (\ref{fga_ansatz_1}) and (\ref{fga_ansatz_2}):
\begin{align}
    \psi^k_{\pm} (\bm{q},\bm{p}) & = \int_{\mathbb{R}^d} \pm \frac{1}{2} \frac{\i}{k c(\bm{q}) |\bm{p}|} f_1^k(\bm{y}) \exp\left(\i k\left(-\bm{p}\cdot(\bm{y} - \bm{q}) + \frac{\i}{2} |\bm{y}-\bm{q}|^2\right)\right) \mathrm{d}\bm{y} \nonumber \\
    & = \pm \frac{1}{2} \frac{\i}{c(\bm{q}) |\bm{p}|} \int_{\mathbb{R}^d} a_{\mathrm{in}}(\bm{y}) \exp\left(\i k S_{\mathrm{in}}(\bm{y})\right) \exp\left(\i k\left(-\bm{p}\cdot(\bm{y}- \bm{q}) + \frac{\i}{2}|\bm{y}-\bm{q}|^2\right)\right) \mathrm{d}\bm{y}.
\end{align}
and define
\begin{align}
    \tilde{\psi}^k (\bm{q},\bm{p}) = \int_{\mathbb{R}^d} a_{\mathrm{in}}(\bm{y}) \exp\left(\i k S_{\mathrm{in}}(\bm{y})\right) \exp\left(\i k\left(-\bm{p}\cdot(\bm{y}- \bm{q}) + \frac{\i}{2}|\bm{y}-\bm{q}|^2\right)\right) \mathrm{d}\bm{y}.
\end{align}
Writing $\tilde{\psi}^k$ as an asymptotic form in the stationary phase method and simplifying it based on Assumption (\ref{ass:inverse}) and Assumption(\ref{ass:jacobian_bounded}) gives
\begin{align}
    \tilde{\psi}^k (\bm{q},\bm{p}) & = \int_{\mathbb{R}^d} \left(a_{\mathrm{in}}(\bm{y}) \exp\left(- \frac{k}{2}|\bm{y} - \bm{q}|^2\right)\right) \exp\left(\i k\left(S_{\mathrm{in}}(\bm{y}) - \bm{p}\cdot(\bm{y} - \bm{q})\right)\right) \mathrm{d} \bm{y} \nonumber \\
    & = (2\pi)^{\frac{d}{2}} k^{-\frac{d}{2}} \chi_{\mathcal{D}(T)}(\bm{p}) \frac{\exp\left(\i k(S_{\mathrm{in}}(\bm{y}) - \bm{p}\cdot(\bm{y} - \bm{q})) + \frac{\i \pi}{4} \mathrm{sgn}(\nabla^2 S_{\mathrm{in}}(\bm{y}))\right)}{|\mathrm{det} \nabla^2 S_{\mathrm{in}}(\bm{y})|^{\frac{1}{2}}} \nonumber \\
    & \quad \quad \times \left(g(\bm{y},\bm{q}) + \sum_{n=1}^\infty k^{-n} \bm{c}_n \cdot D^{2n}_{\bm{y}} g(\bm{y},\bm{q}) \right),
\end{align}
where $g(\bm{y},\bm{q}) = a_{\mathrm{in}}(\bm{y}) \exp\left(-\frac{k}{2}|\bm{y}-\bm{q}|^2\right)$ and $\bm{c}_n(n=1,2,\dots)$ are constant vectors that do not depend on $k$. We define
\begin{align}
    \tilde{\psi}^k_{sp}(\bm{q},\bm{p}) = (2\pi)^{\frac{d}{2}} k^{-\frac{d}{2}} \chi_{\mathcal{D}(T)}(\bm{p}) g(\bm{y},\bm{q}) \frac{\exp\left(\i k(S_{\mathrm{in}}(\bm{y}) - \bm{p}\cdot(\bm{y} - \bm{q})) + \frac{\i \pi}{4} \mathrm{sgn}(\nabla^2 S_{\mathrm{in}}(\bm{y}))\right)}{|\mathrm{det} \nabla^2 S_{\mathrm{in}}(\bm{y})|^{\frac{1}{2}}}.
\end{align}
Then
\begin{align}
    \tilde{\psi}^k(\bm{q},\bm{p}) = \tilde{\psi}^k_{sp}(\bm{q},\bm{p}) \times \left(1 + \sum_{n=1}^\infty k^{-n}\frac{\bm{c}_n\cdot D_{\bm{y}}^{2n} g(\bm{y},\bm{q})}{g(\bm{y},\bm{q})}\right).
\end{align}
Considering the initial amplitude $a(0,\bm{q},\bm{p})$ and taking the moduli of $a(0,\bm{q},\bm{p}) \tilde{\psi}^k(\bm{q},\bm{p})$ and $a(0,\bm{q},\bm{p}) \tilde{\psi}_{sp}^k(\bm{q},\bm{p})$, we obtain
\begin{align} \label{eq:mod_1}
    |a(0,\bm{q},\bm{p}) \tilde{\psi}_{sp}^k(\bm{q},\bm{p})| = 2^d \pi^{\frac{d}{2}} k^{-\frac{d}{2}}  \frac{\chi_{\mathcal{D}(T)}(\bm{p})g(\bm{y},\bm{q})}{|\mathrm{det} \nabla^2 S_{\mathrm{in}}(\bm{y})|^{\frac{1}{2}}},
\end{align}
and
\begin{align} \label{eq:mod_2}
    |a(0,\bm{q},\bm{p}) \tilde{\psi}^k(\bm{q},\bm{p})| = 2^d \pi^{\frac{d}{2}} k^{-\frac{d}{2}}  \frac{\chi_{\mathcal{D}(T)}(\bm{p})g(\bm{y},\bm{q})}{|\mathrm{det} \nabla^2 S_{\mathrm{in}}(\bm{y})|^{\frac{1}{2}}} \left|1 + \sum_{n=1}^\infty k^{-n}\frac{\bm{c}_n\cdot D_{\bm{y}}^{2n} g(\bm{y},\bm{q})}{g(\bm{y},\bm{q})}\right|.
\end{align}

We choose a probability density function as follows:
\begin{align} \label{eq:pdf_wkb}
    \pi(\bm{q},\bm{p}) & = \frac{1}{\mathcal{Z}} \frac{|a(0,\bm{q},\bm{p})\tilde{\psi}^k_{sp}(\bm{q},\bm{p})|}{|\mathrm{det} \nabla^2 S_{\mathrm{in}}(\bm{y})|^{\frac{1}{2}}} \nonumber \\
    & = \frac{1}{\mathcal{Z}} 2^d \pi^{\frac{d}{2}} k^{-\frac{d}{2}} \chi_{\mathcal{D}(T)}(\bm{p}) a_{\mathrm{in}}(\bm{y}) \frac{\exp\left(- \frac{k}{2}|\bm{y} - \bm{q}|^2\right)}{|\mathrm{det} \nabla^2 S_{\mathrm{in}}(\bm{y})|},
\end{align}
where $\mathcal{Z}$ is a normalization parameter to ensure $\pi$ is a probability density function.
We shall use the probability density in (\ref{eq:pdf_wkb}) to generate initial samples for each wave branch, that is, $\pi_{\pm} = \pi$.

The following proposition gives a practical method to sample from the probability density $\pi (\bm{q}, \bm{p})$ given in (\ref{eq:pdf_wkb}). This sampling strategy was first proposed and proved in \cite{Xie2021}. We state it here for convenience. 

{
\begin{proposition} \label{pro:pdf_wkb}
    Consider the wave equation (\ref{wave_eq}) up to time $t$ with the normalized WKB initial conditions (\ref{eq:init_wkb}) that satisfy Assumption (\ref{ass:inverse}), (\ref{ass:inverse_smooth}) and (\ref{ass:jacobian_bounded}).
    Suppose that the random variables $\mathcal{Q}$ and $\mathcal{P}$ are sampled from the probability density $\pi \left( \bm{q}, \bm{p} \right)$ given in (\ref{eq:pdf_wkb}), then $\mathcal{Q}$ and $\mathcal{P}$ satisfy:
    \begin{itemize}
        \item[(1)] The marginal distribution of $\mathcal{Q}$ is a distribution with probability density given as
        \begin{align} \label{eq:pdf_q}
        \pi_{\mathcal{Q}}(\bm{q}) = \frac{2^d \pi^{\frac{d}{2}} k^{-\frac{d}{2}}}{\mathcal{Z}} \int_{\mathrm{R}^d} a_{\mathrm{in}}(\bm{y}) \exp\left(-\frac{k}{2} \left|\bm{y} - \bm{q}\right|^2\right) \d \bm{y}
        \end{align}
        where $\bm{y} = T(\bm{p})$, $T$ is the inverse of $\nabla S_{\mathrm{in}}$.
        \item[(2)] The conditional distribution of $\mathcal{Y}=T(\mathcal{P})$ with respect to $\mathcal{Q}$ is a distribution with probability density given as
        \begin{align} \label{eq:pdf_y}
        \pi_{\mathcal{Y}|\mathcal{Q}}(\bm{y}) = \frac{2^d \pi^{\frac{d}{2}} k^{-\frac{d}{2}}}{\mathcal{Z} \, \pi_{\mathcal{Q}}(\bm{q})} a_{\mathrm{in}}(\bm{y}) \exp\left(-\frac{k}{2}\left|\bm{y}-\bm{q}\right|^2\right)
        \end{align}
        \item[(3)] When $a_{\mathrm{in}} $ is Gaussian given as \begin{align}\label{eq:wkb_gauss_amplitude}
        a_{\mathrm{in}}(\bm{x}) = \left(\prod_{j=1}^d a_j\right)^{\frac{1}{4}} \pi^{-\frac{d}{4}} \exp\left(-\frac{1}{2}\sum_{j=1}^d a_j(x_j-\tilde{x}_j)^2\right)
        \end{align}
        where $a_j > 0$ are independent of $k$. The mean of $a_{\mathrm{in}}$ is $\tilde{\bm{x}} = (\tilde{x}_1,\tilde{x}_2,\cdots,\tilde{x}_d)^T$. Then the marginal distribution of $\mathcal{Q}$ is a normal distribution $\mathcal{N}(\mu_1, \Sigma_1)$ where the expectation $\mu_1 = \tilde{\bm{x}}$ and the covariance matrix is given by
        \begin{align} \label{eq:wkb_sigma1}
            \Sigma_1 = \mathrm{diag} \left( \frac{1}{k} + \frac{1}{a_1}, \frac{1}{k} + \frac{1}{a_2}, \cdots, \frac{1}{k} + \frac{1}{a_d} \right),
        \end{align}
        The conditional distribution of $\mathcal{Y}$ with respect to $\mathcal{Q}$ is a normal distribution $\mathcal{N}(\mu_2, \Sigma_2)$ where the expectation vector and the covariance matrix is given by
        \begin{align} \label{eq:wkb_mu2}
            \mu_2 = \left( \frac{a_1 \tilde{x}_1 + k q_1}{a_1 + k}, \frac{a_2 \tilde{x}_2 + k q_2}{a_2 + k}, \cdots, \frac{a_d \tilde{x}_d + k q_d}{a_d + k} \right),
        \end{align}
        and 
        \begin{align} \label{eq:wkb_sigma2}
            \Sigma_2 = \mathrm{diag} \left( \frac{1}{a_1 + k}, \frac{1}{a_2 + k}, \cdots , \frac{1}{a_d + k} \right),
        \end{align}
        The normalization parameter $\mathcal{Z}$ in (\ref{eq:pdf_wkb}) is given by
        \begin{align} \label{eq:wkb_nml}
            \mathcal{Z} = \left(\prod_{j=1}^d \frac{1}{a_j}\right)^{\frac{1}{4}} 2^{2d} \pi^{\frac{5d}{4}} k^{-d}.
        \end{align}
    \end{itemize}
\end{proposition}
}
According to Proposition \ref{pro:pdf_wkb}, the initial sampling for WKB initial data cases (\ref{eq:init_wkb}) with amplitude (\ref{eq:wkb_gauss_amplitude}) can be formulated into two steps: First, generate $M$ random samples $\left\{ \bm{q}^{(j)} \right\}_{j=1}^M$ from the normal distribution $\mathcal{N}(\mu_1, \Sigma_1)$, where $\mu_1 = \tilde{\bm{x}}$ and $\Sigma_1$ given in (\ref{eq:wkb_sigma1}). Next, for each $\bm{q}^{(j)}$, generate a random sample $\bm{y}^{(j)}$ from the normal distribution $\mathcal{N}(\mu_2, \Sigma_2)$, where $\mu_2$ is given in (\ref{eq:wkb_mu2}) and $\Sigma_2$ is given in (\ref{eq:wkb_sigma2}). Then compute $\bm{p}^{(j)} = \nabla S_{\mathrm{in}} \left( \bm{y}^{(j)} \right)$. As Proposition \ref{pro:pdf_wkb} shows, $z_0^{(j)} = \left( \bm{q}^{(j)}, \bm{p}^{(j)} \right), j = 1,2, \cdots, M$ are independent random samples that follow the probability density $\pi \left( \bm{q}, \bm{p} \right)$ in (\ref{eq:pdf_wkb}). 
{
    When $a_{\mathrm{in}}(\bm{x})$ is not Gaussian, Proposition \ref{pro:pdf_wkb} (3) might not hold. However, one may still sample from the density $\pi(\bm{q},\bm{p})$ given in \ref{eq:pdf_wkb} by other sampling techniques. For instance, for one-dimensional cases, one can calculate the inverse of the marginal distribution function of $q$ and use the inverse transform sampling to generate samples: first generate samples from the uniform distribution on the interval $[0,1]$ and then convert them by the inverse distribution function to obtain the samples that obey the density $\pi_{Q}(q)$. Given any sample $q$, one then may use the inverse transform sampling again to generate samples for $p$. 
}

The following theorem establishes a rigorous error estimate for the FGS in WKB initial data cases. As in the Gaussian initial data cases, the proof of this theorem is also applicable to wave equations with dimensionality $d \geq 3$ and we will show in Section \ref{sec:num_ex} that there are similar results for wave equations with dimensionality $d = 1$ or $d = 2$ through numerical examples.

\begin{theorem} \label{thm:wkb_init}
    Consider the $d \geq 3$ dimensional wave equation (\ref{wave_eq}) up to time $t$ with the normalized WKB initial conditions (\ref{eq:init_wkb}) that satisfy Assumption (\ref{ass:inverse}), (\ref{ass:inverse_smooth}) and (\ref{ass:jacobian_bounded}) and $a_{\mathrm{in}}$ is given in (\ref{eq:wkb_gauss_amplitude}). Suppose that the velocity function $c(\bm{x})$ is bounded in $\mathbb{R}^d$ with its infimum $c_{\mathrm{inf}} = \inf_{\bm{x}\in \mathbb{R}^d} c(\bm{x}) > 0$. When applying the frozen Gaussian sampling (FGS) with sampling density functions $\pi_{\pm}$ as in (\ref{eq:pdf_wkb}), there exists a positive constant $C(t,d)$ that is independent of $k$ such that
    \begin{align} \label{eq:wkb_es}
        \mathcal{E}_S \leq \frac{C(t,d)}{M}k^{\frac{d}{2}},
    \end{align}
    and
    \begin{equation} \label{eq:wkb_e0}
        \mathcal{E}_0 \leq (C_A(t) k^{-1})^2 + \frac{C(t,d)}{M}k^{\frac{d}{2}}.
    \end{equation}
\end{theorem}

\begin{proof} Omit the $\pm$ subscript in the following derivation. By Lemma (\ref{lm:frame_error_anaz}) (2) and Lemma (\ref{lm:a_bounded}),
\begin{align} \label{ieq:est_wkb}
    & \frac{1}{k^2} \int_{\mathbb{R}^d} \mathbb{E}|\partial_t\Lambda(t,\bm{x},z_0)|^2 \mathrm{d}\bm{x} + \frac{1}{k^2} \int_{\mathbb{R}^d} \mathbb{E}|\nabla_{\bm{x}} \Lambda(t,\bm{x},z_0)|^2 \mathrm{d}\bm{x} \nonumber \\
    & \hspace{1em} \leq \frac{C_0(t,d)}{2^{3d} (\pi/k)^{\frac{5d}{2}}} \int_{\mathbb{R}^{2d}} \frac{|a(t,z_0)\psi(z_0)|^2}{\pi(z_0)} \mathrm{d}z_0 \nonumber \\
    & \hspace{1em} \leq C_0(t,d) (K(t,d))^2 \frac{\mathcal{Z}}{2^{3d} (\pi/k)^{\frac{5d}{2}}} \int_{\mathbb{R}^{2d}} \frac{1}{4} \frac{1}{c^2(\bm{q})|\bm{p}|^2} \left|a(0,z_0)\right| \frac{|\tilde{\psi}(z_0)|^2}{|\tilde{\psi}_{sp}(z_0)|} |\mathrm{det} \nabla^2 S_{\mathrm{in}}(\bm{y})|^{\frac{1}{2}} \mathrm{d}z_0.
\end{align}
Recall that $a_{\mathrm{in}}$ is given by (\ref{eq:wkb_gauss_amplitude}), which is a Gaussian amplitude that does not depend on $k$, and \[g(\bm{y}, \bm{q}) = a_{\mathrm{in}}(\bm{y}) \exp\left(-\frac{k}{2} |\bm{y}-\bm{q}|^2\right).\] Therefore, taking the derivatives of $g(\bm{y},\bm{q})$ of order $2n$ with respect to $\bm{y}$ yields
\begin{equation}
    \bm{c}_n \cdot D^{2n}_{\bm{y}} g(\bm{y},\bm{q}) = \mathcal{O}\left(k^{2n} |\bm{y}-\bm{q}|^{2n}\right) g(\bm{y},\bm{q}),
\end{equation}
where $\bm{c}_n$ are constant vectors given in Lemma \ref{lm:stat_phase}. 
Therefore, according to (\ref{eq:mod_1}) and (\ref{eq:mod_2}), 
\begin{align} \label{eq:prf_h_2n}
    \frac{|\tilde{\psi}(z_0)|^2}{|\tilde{\psi}_{sp}(z_0)|} = |\tilde{\psi}_{sp}(z_0)| \times \left(1 + \sum_{n=1}^\infty \mathcal{O} \left(k^{n} |\bm{y} - \bm{q}|^{2n}\right)\right).
\end{align}

First, estimate the integral:
\begin{align}
    I = \int_{\mathbb{R}^{2d}} \frac{1}{4} \frac{1}{c^2(\bm{q})|\bm{p}|^2} \left|a(0,z_0)\tilde{\psi}_{sp}(z_0)\right| |\mathrm{det} \nabla^2 S_{\mathrm{in}}(\bm{y})|^{\frac{1}{2}} \mathrm{d}z_0.
\end{align}
Denote the integrand of $I$ by $I(\bm{q},\bm{p})$, then 
\begin{align}
    I(\bm{q},\bm{p}) & = \frac{1}{4} \frac{1}{c^2(\bm{q})|\bm{p}|^2} \left|a(0,z_0)\tilde{\psi}_{sp}(z_0)\right| |\mathrm{det} \nabla^2 S_{\mathrm{in}}(\bm{y})|^{\frac{1}{2}} \nonumber \\
    & = \frac{1}{4} \frac{1}{c^2(\bm{q})|\bm{p}|^2} 2^d \pi^{\frac{d}{2}} k^{-\frac{d}{2}} \chi_{\mathcal{D}(T)}(\bm{p}) a_{\mathrm{in}}(\bm{y}) \exp\left(-\frac{k}{2}|\bm{y}-\bm{q}|^2\right) \nonumber \\
    & = \left(\prod_{j=1}^d a_j\right)^{\frac{1}{4}} \frac{1}{c^2(\bm{q})|\bm{p}|^2}2^{d-2} \pi^{\frac{d}{4}} k^{-\frac{d}{2}} \chi_{\mathcal{D}(T)}(\bm{p}) \exp\left(-\frac{1}{2}\sum_{j=1}^d a_j \left(y_j - \tilde{x}_j\right)^2 - \frac{k}{2}|\bm{y}-\bm{q}|^2\right) \nonumber \\
    & \leq \tilde{C}_0 k^{-\frac{d}{2}} \frac{1}{|\bm{p}|^2} \chi_{\mathcal{D}(T)}(\bm{p}) \exp\left(-\frac{1}{2}\sum_{j=1}^d a_j \left(y_j - \tilde{x}_j\right)^2 - \frac{k}{2}|\bm{y}-\bm{q}|^2\right) := \tilde{I}(\bm{q},\bm{p}).
\end{align}
where we denote
\begin{align}
    \tilde{C}_0 = \left(\prod_{j=1}^d a_j\right)^{\frac{1}{4}} c_{\mathrm{inf}}^{-2} 2^{d-2} \pi^{\frac{d}{4}}.
\end{align}
Therefore,
\begin{align}
    I = \int_{\mathbb{R}^{2d}} I(\bm{q},\bm{p}) \mathrm{d}\bm{q}\mathrm{d}\bm{p} \leq \int_{\mathbb{R}^{2d}} \tilde{I}(\bm{q},\bm{p}) \mathrm{d}\bm{q}\mathrm{d}\bm{p} := \tilde{I}.
\end{align}
Notice that
\begin{align}
    \tilde{I} = \tilde{C}_0 k^{-\frac{d}{2}} \int_{\mathbb{R}^d \times \mathcal{D}(T)} \frac{1}{|\bm{p}|^2} \exp\left(- \frac{1}{2}\sum_{j=1}^d a_j\left(y_j-\tilde{x}_j\right)^2 - \frac{k}{2}|\bm{y}-\bm{q}|^2\right) \mathrm{d}\bm{q}\mathrm{d}\bm{p}.
\end{align}
Change the variables from $(\bm{q},\bm{p})$ to $(\bm{q},\bm{y})$ so that $\tilde{I}$ can be writen as 
\begin{align}
    \tilde{I} = \tilde{C}_0 k^{-\frac{d}{2}} \int_{\mathbb{R}^{2d}} \frac{1}{|\nabla S_{\mathrm{in}}(\bm{y})|^2} \exp\left(- \frac{1}{2}\sum_{j=1}^d a_j\left(y_j-\tilde{x}_j\right)^2 - \frac{k}{2}|\bm{y}-\bm{q}|^2\right) |\mathrm{det}\nabla^2 S_{\mathrm{in}}(\bm{y})|\mathrm{d}\bm{q}\mathrm{d}\bm{y}.
\end{align}
Using Assumption (\ref{ass:inverse_smooth}) and (\ref{ass:jacobian_bounded}), we get
\begin{align}
    \tilde{I} \leq \tilde{C}_0 k^{-\frac{d}{2}} C^* M_T^2 \int_{\mathbb{R}^{2d}} \frac{1}{|\bm{y}|^2} \exp\left(- \frac{1}{2}\sum_{j=1}^d a_j\left(y_j-\tilde{x}_j\right)^2 - \frac{k}{2}|\bm{y}-\bm{q}|^2\right) \mathrm{d}\bm{q}\mathrm{d}\bm{y}.
\end{align}
Now the key is to estimate the following integral, which is denoted by 
\begin{align}
    \hat{I} := \int_{\mathbb{R}^{2d}} \hat{I}(\bm{q},\bm{y}) \mathrm{d}\bm{q}\mathrm{d}\bm{y} = \int_{\mathbb{R}^{2d}} \frac{1}{|\bm{y}|^2} \exp\left(- \frac{1}{2}\sum_{j=1}^d a_j\left(y_j-\tilde{x}_j\right)^2 - \frac{k}{2}|\bm{y}-\bm{q}|^2\right) \mathrm{d}\bm{q}\mathrm{d}\bm{y},
\end{align}
where 
\begin{align}
    \hat{I}(\bm{q},\bm{y}) := \frac{1}{|\bm{y}|^2} \exp\left(- \frac{1}{2}\sum_{j=1}^d a_j\left(y_j-\tilde{x}_j\right)^2 - \frac{k}{2}|\bm{y}-\bm{q}|^2\right).
\end{align}
First, compute that
\begin{align}
    \int_{\mathbb{R}^d} \hat{I}(\bm{q},\bm{y}) \mathrm{d}\bm{q} = 2^{\frac{d}{2}} k^{-\frac{d}{2}} \frac{1}{|\bm{y}|^2} \exp\left(- \frac{1}{2}\sum_{j=1}^d a_j\left(y_j-\tilde{x}_j\right)^2\right).
\end{align}
Then take $\tilde{a} = \min_{1\leq j \leq d} a_j$. Since $a_j > 0$ for each $j$, $\tilde{a} > 0$.
\begin{align}
    \hat{I} & = \int_{\mathbb{R}^d} \mathrm{d}\bm{y} \int_{\mathbb{R}^d} \hat{I}(\bm{q},\bm{y}) \mathrm{d}\bm{q} = 2^{\frac{d}{2}} k^{-\frac{d}{2}} \int_{\mathbb{R}^d} \frac{1}{|\bm{y}|^2} \exp\left(- \frac{1}{2}\sum_{j=1}^d a_j\left(y_j-\tilde{x}_j\right)^2\right) \mathrm{d}\bm{y} \nonumber \\
    & \leq 2^{\frac{d}{2}} k^{-\frac{d}{2}} \int_{\mathbb{R}^d} \frac{1}{|\bm{y}|^2} \exp\left(-\frac{\tilde{a}}{2}|\bm{y} - \tilde{\bm{x}}|^2\right) \mathrm{d}\bm{y} := 2^{\frac{d}{2}} k^{-\frac{d}{2}} \tilde{I}_y.
\end{align}
Take $h = |\tilde{\bm{x}}| = \left(\sum_{j=1}^d \tilde{x}_j^2\right)^\frac{1}{2}$. Without loss of generality, we assume that $\tilde{\bm{x}} \ne 0$ so that $h > 0$. If $\tilde{\bm{x}} = 0$, we rotate to new coordinates to make $\tilde{\bm{x}}$ nonzero and proceed. Split the integral $\tilde{I}_y$ into two parts, that is,
\begin{align}
    \tilde{I}_y = \left(\int_{B(0,\frac{h}{2})} + \int_{\mathbb{R}^d \backslash B(0,\frac{h}{2})}\right) \frac{1}{|\bm{y}|^2} \exp\left(-\frac{\tilde{a}}{2}|\bm{y} - \tilde{\bm{x}}|^2\right) \mathrm{d}\bm{y} := \tilde{I}_{y,1} + \tilde{I}_{y,2}.
\end{align}
For the first part,
\begin{align} \label{ieq:i_y1}
    \tilde{I}_{y,1} 
    & \leq \int_{B(0,\frac{h}{2})} \frac{1}{|\bm{y}|^2} \exp\left(-\frac{\tilde{a}}{2} \frac{h^2}{4}\right) \mathrm{d}\bm{y}
    \nonumber\\
    & = \exp\left(-\frac{\tilde{a}h^2}{8}\right) 
    \int_0^{\frac{h}{2}} \int_0^\pi \cdots \int_0^\pi \int_0^{2\pi} r^{d-3} \sin^{d-2}\phi_1  \cdots \sin^{2} \phi_{d-3} \sin \phi_{d-2} \d r \d \phi_1 \cdots \d \phi_{n-2} \d \phi_{n-1},
\end{align}
where spherical coordination has been used. As mentioned in the proof of Theorem \ref{thm:gauss_init}, as long as $d \geq 3$, the integral on the right-hand side is convergent and its value is independent of $k$. We also use the fact that there is a bound $C_d$ such that $\exp\left(-x\right) \leq C_d x^{-\frac{d}{2}}$ for any positive $d$ and for all $x > 0$. Consequently, 
\begin{align}
    \tilde{I}_{y,1}  \leq \tilde{C}_d \left(\frac{\tilde{a}h^2}{8} \right)^{-\frac{d}{2}} = \tilde{C}_d 2^{\frac{3d}{2}} h^{-d} \tilde{a}^{-\frac{d}{2}},
\end{align}
where $\tilde{C}_d$ is a constant depending on $d$. For the second part,
\begin{align} \label{ieq:i_y2}
    \tilde{I}_{y,2} & = \int_{\mathbb{R}^d\backslash B(0,\frac{h}{2})} \frac{1}{|\bm{y}|^2} \exp\left(-\frac{\tilde{a}}{2}|\bm{y} - \tilde{\bm{x}}|^2\right) \mathrm{d}\bm{y} 
    \leq \int_{\mathbb{R}^d\backslash B(0,\frac{h}{2})} \frac{4}{h^2} \exp\left(-\frac{\tilde{a}}{2}|\bm{y} - \tilde{\bm{x}}|^2\right) \mathrm{d}\bm{y} \nonumber \\
    & \leq \int_{\mathbb{R}^d} \frac{4}{h^2} \exp\left(-\frac{\tilde{a}}{2}|\bm{y} - \tilde{\bm{x}}|^2\right) \mathrm{d}\bm{y} = 2^{\frac{d}{2}+2} \pi^{\frac{d}{2}} h^{-2} \tilde{a}^{-\frac{d}{2}}.
\end{align}
Combining the results of (\ref{ieq:i_y1}) and (\ref{ieq:i_y2}), we obtain
\begin{align}
    \hat{I} \leq 2^{\frac{d}{2}} k^{-\frac{d}{2}} \tilde{I}_y = 2^{\frac{d}{2}} k^{-\frac{d}{2}} (\tilde{I}_{y,1} + \tilde{I}_{y,2}) \leq \left(2^{2d} \tilde{C}_d h^{-d} + 2^{d+2} \pi^{\frac{d}{2}} h^{-2}\right) \tilde{a}^{-\frac{d}{2}} k^{-\frac{d}{2}}.
\end{align}
Therefore,
\begin{align}
    I \leq \tilde{I} & \leq \tilde{C}_0 k^{-\frac{d}{2}} C^* M_T^2 \hat{I} 
    \leq \left(\prod_{j=1}^d a_j\right)^{\frac{1}{4}} c_{\mathrm{inf}}^{-2} \left(2^{2d+1} \pi^{\frac{d}{4}} \tilde{C}_d h^{-d} + 2^{d+3} \pi^{\frac{3d}{4}} h^{-2} \right) C^* M_T^2 \tilde{a}^{-\frac{d}{2}} k^{-d}.
\end{align}
Denote
\begin{equation} \label{eq:wkb_c0}
    C_0^* = \left(\prod_{j=1}^d a_j\right)^{\frac{1}{4}} c_{\mathrm{inf}}^{-2} \left(2^{2d+1} \pi^{\frac{d}{4}} \tilde{C}_d h^{-d} + 2^{d+3} \pi^{\frac{3d}{4}} h^{-2} \right) C^* M_T^2 \tilde{a}^{-\frac{d}{2}}.
\end{equation}
Then
\begin{align} \label{eq:int_ord1}
    I = \int_{\mathbb{R}^{2d}} \frac{1}{4} \frac{1}{c^2(\bm{q})|\bm{p}|^2} \left|a(0,z_0)\tilde{\psi}_{sp}(z_0)\right| |\mathrm{det} \nabla^2 S_{\mathrm{in}}(\bm{y})|^{\frac{1}{2}} \mathrm{d}z_0 \leq C_0^* k^{-d},
\end{align}
where $C_0^*$ is a constant that is independent of $k$, that is, $I$ converges and has a bound of order $\mathcal{O}(k^{-d})$.

Now consider the integral
\begin{align} 
    \int_{\mathbb{R}^{2d}} \frac{1}{4} \frac{1}{c^2(\bm{q})|\bm{p}|^2} \left|a(0,z_0)\tilde{\psi}_{sp}(z_0)\right| |\mathrm{det} \nabla^2 S_{\mathrm{in}}(\bm{y})|^{\frac{1}{2}} \times k^n |\bm{y} - \bm{q}|^{2n} \mathrm{d}z_0.
\end{align}
Let us write the integrand as:
\begin{align} \label{eq:int_ordn}
    & \frac{1}{4} \frac{1}{c^2(\bm{q})|\bm{p}|^2} \left|a(0,z_0)\tilde{\psi}_{sp}(z_0)\right| |\mathrm{det} \nabla^2 S_{\mathrm{in}}(\bm{y})|^{\frac{1}{2}} \times k^n |\bm{y} - \bm{q}|^{2n} \nonumber\\
    = & \frac{1}{4} \frac{1}{c^2(\bm{q})|\bm{p}|^2} 2^d \pi^{\frac{d}{2}} k^{-\frac{d}{2}} \chi_{\mathcal{D}(T)}(\bm{p}) a_{\mathrm{in}}(\bm{y}) \exp\left(-\frac{k}{2}|\bm{y}-\bm{q}|^2\right) \times  k^n |\bm{y} - \bm{q}|^{2n} \nonumber \\
    = & \frac{1}{4} \frac{1}{c^2(\bm{q})|\bm{p}|^2} 2^d \pi^{\frac{d}{2}} k^{-\frac{d}{2}} \chi_{\mathcal{D}(T)}(\bm{p})  a_{\mathrm{in}}(\bm{y}) \exp\left(-\frac{k}{4}|\bm{y}-\bm{q}|^2\right) \times  k^n |\bm{y} - \bm{q}|^{2n} \exp\left(-\frac{k}{4}|\bm{y}-\bm{q}|^2\right). 
\end{align}
Apply a change of variables $\bar{\bm{y}} = \sqrt{k} \bm{y}$, $\bar{\bm{q}} = \sqrt{k} \bm{q}$ as follows:
\begin{equation} \label{eq:prf_wkb_chg}
    k^n |\bm{y} - \bm{q}|^{2n} \exp\left(-\frac{k}{4}|\bm{y}-\bm{q}|^2\right) = |\bar{\bm{y}} - \bar{\bm{q}}|^{2n} \exp\left(-\frac{1}{4}|\bar{\bm{y}}-\bar{\bm{q}}|^2\right).
\end{equation}
Since the right-hand side of (\ref{eq:prf_wkb_chg}) has a bound that is irrelevant to $k$ for all $(\bar{\bm{q}},\bar{\bm{y}}) \in \mathbb{R}^{2d}$, so has the left-hand side of (\ref{eq:prf_wkb_chg}) a bound that is independent of $k$ for all $(\bm{q},\bm{y}) \in \mathbb{R}^{2d}$, that is, 
\begin{align}
    \left \Vert k^n |\bm{y} - \bm{q}|^{2n} \exp\left(-\frac{k}{4}|\bm{y}-\bm{q}|^2\right) \right \Vert_{L^{\infty}(\mathbb{R}^{2d})} = \left \Vert |\bar{\bm{y}} - \bar{\bm{q}}|^{2n} \exp\left(-\frac{1}{4}|\bar{\bm{y}}-\bar{\bm{q}}|^2\right) \right \Vert_{L^{\infty}(\mathbb{R}^{2d})} 
\end{align}
does not depend on $k$. Moreover, similar to the deduction of the boundedness in (\ref{eq:int_ord1}), 
\begin{align}
    \int_{\mathbb{R}^{2d}} \frac{1}{4} \frac{1}{c^2(\bm{q})|\bm{p}|^2} 2^d \pi^{\frac{d}{2}} k^{-\frac{d}{2}} \chi_{\mathcal{D}(T)}(\bm{p}) a_{\mathrm{in}}(\bm{y}) \exp\left(-\frac{k}{4}|\bm{y}-\bm{q}|^2\right) \d z_0
\end{align}
also has an upper bound of order $\mathcal{O}(k^{-d})$. Therefore, according to (\ref{eq:int_ordn}) we deduce that
\begin{align}
    \int_{\mathbb{R}^{2d}} \frac{1}{4} \frac{1}{c^2(\bm{q})|\bm{p}|^2} \left|a(0,z_0)\tilde{\psi}_{sp}(z_0)\right| |\mathrm{det} \nabla^2 S_{\mathrm{in}}(\bm{y})|^{\frac{1}{2}} \times k^n |\bm{y} - \bm{q}|^{2n} \d z_0 < \infty,
\end{align}
and it has an upper bound of order $\mathcal{O}(k^{-d})$.

Now consider the following:
\begin{align}\label{eq:uni_conv}
    \int_{\mathbb{R}^{2d}} \sum_{n=0}^\infty \left(\frac{1}{4} \frac{1}{c^2(\bm{q})|\bm{p}|^2} \left|a(0,z_0) \tilde{\psi}_{sp}(z_0) \right| \left|\mathrm{det} \nabla^2 S_{\mathrm{in}}(\bm{y})\right|^{\frac{1}{2}} \times k^n |\bm{y} - \bm{q}|^{2n} \right) \mathrm{d}z_0.
\end{align}
As the integrand of (\ref{eq:uni_conv}) converges uniformly, it can be integrated termwise, that is, 
\begin{align}
    & \int_{\mathbb{R}^{2d}} \sum_{n=0}^\infty \left(\frac{1}{4} \frac{1}{c^2(\bm{q})|\bm{p}|^2} \left|a(0,z_0) \tilde{\psi}_{sp}(z_0) \right| \left|\mathrm{det} \nabla^2 S_{\mathrm{in}}(\bm{y})\right|^{\frac{1}{2}} \times k^n |\bm{y} - \bm{q}|^{2n} \right) \mathrm{d}z_0 \nonumber \\
    = & \sum_{n=0}^\infty \int_{\mathbb{R}^{2d}}  \frac{1}{4} \frac{1}{c^2(\bm{q})|\bm{p}|^2} \left|a(0,z_0) \tilde{\psi}_{sp}(z_0) \right| \left|\mathrm{det} \nabla^2 S_{\mathrm{in}}(\bm{y})\right|^{\frac{1}{2}} \times k^n |\bm{y} - \bm{q}|^{2n} \mathrm{d}z_0.
\end{align}
Besides, each term of the series is convergent and has a bound of order $\mathcal{O}(k^{-d})$, which entails that the series itself also has a bound of the same order. 

Consequently, combined with (\ref{eq:prf_h_2n}), there must be a constant $\tilde{C}$ that is irrelevant to $k$ such that 
\begin{align}
    \int_{\mathbb{R}^{2d}} \frac{1}{4} \frac{1}{c^2(\bm{q})|\bm{p}|^2} \left|a(0,z_0)\right| \frac{|\tilde{\psi}(z_0)|^2}{|\tilde{\psi}_{sp}(z_0)|} |\mathrm{det} \nabla^2 S_{\mathrm{in}}(\bm{y})|^{\frac{1}{2}} \mathrm{d}z_0 \leq \tilde{C} k^{-d} .
\end{align} 
Provided that the normalization parameter is given in (\ref{eq:wkb_nml}), it follows from (\ref{ieq:est_wkb}) that 
\begin{align}
    \frac{1}{k^2} \int_{\mathbb{R}^d} \mathbb{E}|\partial_t\Lambda(t,\bm{x},z_0)|^2 \mathrm{d}\bm{x} + \frac{1}{k^2} \int_{\mathbb{R}^d} \mathbb{E}|\nabla_{\bm{x}} \Lambda(t,\bm{x},z_0)|^2 \mathrm{d}\bm{x} & \leq \tilde{C}(t,d) k^{\frac{d}{2}},
\end{align}
where
\begin{equation} \label{eq:wkb_tld_c}
    \tilde{C}(t,d) = C_0(t,d) (K(t,d))^2 \prod_{j=1}^d \left( \frac{1}{a_j} \right)^{\frac{1}{4}} 2^{-d} \pi^{-\frac{5d}{4}} \tilde{C}
\end{equation}
is a positive constant depending on $t$ but not $k$. According to Lemma (\ref{lm:frame_error_anaz}) (1), we obtain
\begin{align}
    \mathbb{E}\left\Vert u_{\mathrm{FGS}}^k\left(t,\cdot \right) - u_{\mathrm{FGA}}^k\left(t,\cdot \right) \right\Vert_E^2 \leq \frac{8 \tilde{C}(t,d)}{M} k^{\frac{d}{2}}.
\end{align}
This proves (\ref{eq:wkb_es}), and (\ref{eq:wkb_e0}) follows directly from (\ref{eq:error_fga}).
\end{proof}

\begin{remark}
    In the above proof, when $a_{\mathrm{in}}$ is not Gaussian, one can obtain similar results as long as $a_{\mathrm{in}}$ satisfies the following
    \begin{align}
        \left|D^{2n} a_{\mathrm{in}}(\bm{x})\right| \leq \left|h^{2n}(\bm{x}) a_{\mathrm{in}}(\bm{x})\right|
    \end{align}
    for any positive integer $n$, where $h^{2n}(\bm{x})$ is a $2n$-degree polynomial of $\bm{x}$.
\end{remark}

\begin{remark}
    We discuss how $C(t,d)$ in (\ref{eq:wkb_es}) and (\ref{eq:wkb_e0}) depends on the dimensionality $d \geq 3$. In the proof of Theorem \ref{thm:wkb_init}, we have $C(t,d) = 8\tilde{C}(t,d)$. From (\ref{eq:wkb_tld_c}), we notice that $\tilde{C}(t,d)$ depends on the constant $\tilde{C}$ and the behavior of $\tilde{C}$ as $d$ increases should be very similar to that of $C_0^*$ given in (\ref{eq:wkb_c0}). Thus it suffices to determine the relationship between $C_0^*$ and $d$. Similar to the Gaussian cases, $C_0^*$ increases exponentially with $d$. Therefore, it is reasonable to conclude that $\tilde{C}(t,d)$, as well as $C(t,d)$, also increases exponentially as $d$ increases.
\end{remark}

{
    
    \begin{remark}
        We remark on the advantages of the FGS over other existing methods such as the finite difference method. For simplicity, we ignore the asymptotic error $E_{\mathrm{FGA}}$ of the FGA ansatz and estimate the approximation error $E_0$ of the FGS algorithm by the sampling error $E_S$, that is,
        \begin{equation}
            E_0 = \Vert u - u_{\mathrm{FGS}} \Vert_E \approx \left(\mathcal{E}_S\right)^{\frac{1}{2}} \leq \sqrt{\frac{C(t,d)}{M}k^{\frac{d}{2}}} .
        \end{equation}
        In order to guarantee that $E_0 \leq \delta$, the sample size of the FGS algorithm should be more than 
        \begin{equation}
            \mathcal{N}_{\mathrm{FGS}} = \frac{C(t,d)}{\delta^2} k^{\frac{d}{2}}, \quad \delta > 0.
        \end{equation}
        {
            Note that the constant $C(t,d)$ in the least sample size $\mathcal{N}_{\mathrm{FGS}}$ is a $d$th power of an $\mathcal{O}(1)$ parameter. Although $\mathcal{N}_{\mathrm{FGS}}$ increases with $k$ in the WKB initial cases, but compared to the computational complexity of mesh-based numerical methods given in (\ref{eq:fd_order}), there is still a big saving on computational costs to solve the wave equation with the FGS.
        }
    \end{remark}
}

\section{Numerical examples} \label{sec:num_ex}

In the previous sections, we present how to apply the FGS algorithm to solve scalar wave equations with Gaussian initial conditions and WKB initial conditions. We also elaborate the relationship between the sampling error, the sample size $M$, the wave number $k$ and the dimensionality $d$. In this section, we give several examples of using the FGS to numerically solve the wave equation (\ref{wave_eq}) and illustrate the theoretical results in Theorem \ref{thm:gauss_init} and Theorem \ref{thm:wkb_init}. These examples include applying the FGS to solve 1D, 2D, and 3D wave equations with Gaussian and WKB initial data respectively. For each example, we explore the relationship of the sampling error and the wave number under different velocity functions, and compare the numerical results with the theoretical conclusions. 

\subsection{Examples with Gaussian initial conditions}

In this subsection, we aim to study the results of the sampling error $E_S$ when applying the FGS to wave equations with Gaussian initial conditions. We use the FGS to approximate the wave function at time $t = 0.5$ and study the relationship between the sampling error $E_S$, the wave number $k$ and the sample size $M$.

Recall that sampling error $E_S$ is defined as
\begin{align}
    E_S(M,k) = \left\Vert u^k_{\mathrm{FGA}}(t,\cdot) - u^k_{\mathrm{FGS}}\left(t,\cdot;M,\left\{ z_0^{(j)} \right\}_{j=1}^M \right) \right\Vert_{E},
\end{align}
where $u^k_{\mathrm{FGA}}$ is the FGA ansatz and the $u^k_{\mathrm{FGS}}$ is the FGS solution. Notice that for any fixed setting of $(M,k)$, the sampling error $E_S$ is a random variable and varies with different initial samples. Therefore, we measure the mean square sampling error $\mathcal{E}_S(M,k)$, that is, 
\begin{align}
    \mathcal{E}_S(M,k) = \mathbb{E} \left\Vert u^k_{\mathrm{FGA}}(t,\cdot) - u^k_{\mathrm{FGS}}\left(t,\cdot;M,\left\{ z_0^{(j)} \right\}_{j=1}^M \right) \right\Vert_{E}^2,
\end{align}
in each numerical case. We use 30 simulations for each sample size $M$ and wave number $k$ and calculate the square root of the mean square sample error. In each case, we use a sufficiently large sample size to compute an FGS wave function as an approximation of the FGA ansatz. For wave equations of dimensionality $d = 1,2,3$, the large sample sizes used are $M_0 = 1.5 \times 10^5, 2.0 \times 10^5, 2.5 \times 10^5$ respectively. In each of the following examples, the initial sampling is based on the density given (\ref{eq:pdf_gauss}), and the fourth order Runge-Kutta method is used to evolve the ODE system (\ref{eq:ode_qp}) and (\ref{eq:ode_a}) for each initial sample up to time $t = 0.5$. The time step is chosen as $\Delta t = 1.0 \times  10^{-4}$, which is small enough that the error of the Runge-Kutta method is negligible. The FGS wave function is reconstructed based on (\ref{eq:fgs_wavefunc}). $\psi^k_{\pm}$ is computed directly by (\ref{eq:psi_gauss}).

\begin{xmpl}
Use the FGS to solve 1D wave equation. The initial conditions are
\begin{equation}
    \begin{cases}
        f_0^k(x) = 0, \\
        f_1^k(x) = k \left( \frac{2k}{\pi}\right)^{\frac{1}{4}} \exp\left(-\i k x -k x^2\right).
    \end{cases}
\end{equation}
The velocity functions are: $c_1(x) = 1$ and $c_2(x) = 1+\frac{1}{4}\sin(x)$. 
\end{xmpl}

We choose wave numbers $k = 2^9, 2^{10}, 2^{11}, 2^{12}$ and sample sizes $M$ = 50, 100, 200, 400, 800, 1600, 3200. Initial samples $\{z_0^{(j)}\}_{j=1}^M$ are obtained from the following probability density: 
\begin{align}
    \pi(q,p) = \left(\frac{k}{2\pi}\right) \frac{\sqrt{2}}{3} \exp\left(-\frac{(p+1)^2 + 2q^2}{6/k}\right).
\end{align}
The wave function is constructed on the spatial interval $[-1,1]$. 

We present the numerical results in Table \ref{tb:1d_gauss} and Figure \ref{fig:1d_gauss}. For both velocity functions, the mean square sampling errors $\mathcal{E}_S$ are independent of $k$ and their square roots are of order $\mathcal{O}(M^{-\frac{1}{2}})$, which indicates that there are similar results as Theorem (\ref{thm:gauss_init}) when applying the FGS to 1D wave equations with Gaussian initial data.  
\begin{table}[htb]
	\centering
        \caption{The square root of the mean square sampling error $\mathcal{E}_S(M,k)$ of applying the FGS to 1D wave equations with Gaussian initial data.}
	\begin{tabular}{c|ccccccc}
		\toprule
        $c_1(x)$ & $M=50$ & $M=100$ & $M=200$ & $M=400$ & $M=800$ & $M=1600$ & $M=3200$ \\
        \midrule
        $k=2^9$ & 2.3856e-01 & 1.7932e-01 & 1.1703e-01 & 8.8621e-02 & 6.2476e-02 & 4.4833e-02 & 3.2010e-02 \\
        $k=2^{10}$ & 2.4038e-01 & 1.8117e-01 & 1.1775e-01 & 9.3241e-02 & 6.2403e-02 & 4.5508e-02 & 3.0250e-02 \\
        $k=2^{11}$ & 2.6065e-01 & 1.7205e-01 & 1.2661e-01 & 8.5084e-02 & 6.2171e-02 & 4.1782e-02 & 3.2901e-02 \\
        $k=2^{12}$ & 2.5090e-01 & 1.8182e-01 & 1.2389e-01 & 9.1502e-02 & 6.6161e-02 & 4.6389e-02 & 3.1538e-02 \\
		\midrule
        $c_2(x)$ & $M=50$ & $M=100$ & $M=200$ & $M=400$ & $M=800$ & $M=1600$ & $M=3200$ \\
        \midrule
        $k=2^9$ & 2.4662e-01 & 1.7574e-01 & 1.2691e-01 & 9.4593e-02 & 6.1579e-02 & 4.8843e-02 & 3.2599e-02 \\
        $k=2^{10}$ & 2.4911e-01 & 1.8382e-01 & 1.1948e-01 & 9.4251e-02 & 6.0482e-02 & 4.4729e-02 & 3.1154e-02 \\
        $k=2^{11}$ & 2.5076e-01 & 1.8512e-01 & 1.2300e-01 & 9.2491e-02 & 6.2701e-02 & 4.4771e-02 & 3.0680e-02 \\
        $k=2^{12}$ & 2.5542e-01 & 1.8720e-01 & 1.2539e-01 & 8.5601e-02 & 6.1707e-02 & 4.4333e-02 & 3.3418e-02 \\
		\bottomrule
	\end{tabular}
    \label{tb:1d_gauss}
\end{table}
\begin{figure}[htb] 
    \centering 
    \includegraphics[width=1\textwidth,trim={50 10 50 10},clip]{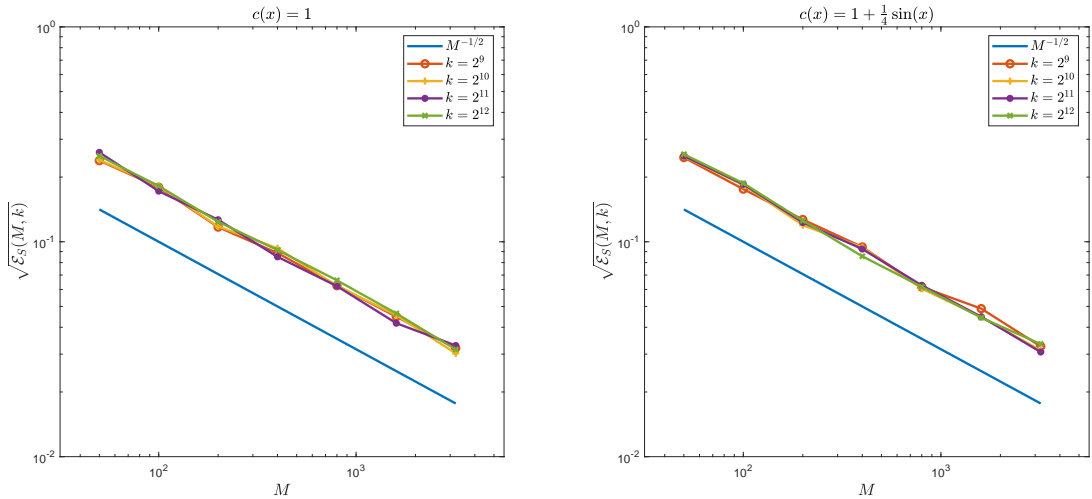} 
    \caption{The square root of the mean square sampling error $\mathcal{E}_S(M,k)$ of applying the FGS to 1D wave equations with Gaussian initial data.} 
    \label{fig:1d_gauss} 
\end{figure}

\begin{xmpl}
Use the FGS to solve 2D wave equation. The initial conditions are
\begin{equation}
    \begin{cases}
        f_0^k(\bm{x}) = 0, \\
        f_1^k(\bm{x}) = k \left( \frac{2k}{\pi}\right)^{\frac{1}{2}} \exp\left( - \i k(x_1 +x_2) - k \left| \bm{x} \right|^2\right).
    \end{cases}
\end{equation}
The velocity functions are: $c_1(\bm{x}) = 1$ and $c_2(\bm{x}) = 1 + \frac{1}{4} \sin(x_1 + x_2)$.
\end{xmpl}

We choose wave numbers $k = 2^8, 2^9, 2^{10}, 2^{11}$ and sample sizes $M$ = 200, 400, 800, 1600, 3200, 6400, 12800. Initial samples $\{z_0^{(j)}\}_{j=1}^M$ are from the following density: 
\begin{equation}
    \pi(\bm{q},\bm{p}) = \left(\frac{k}{2\pi}\right)^2 \frac{2}{9} \exp\left(-\frac{(p_1+1)^2 + (p_2+1)^2 + 2(q_1^2 + q_2^2)}{6/k}\right).
\end{equation}
The wave function is constructed on the spatial interval $[-1,1] \times[-1,1]$. 

We present the numerical results in Figure \ref{fig:2d_gauss}. One may easily notice that for each velocity function, the mean square sampling error $\mathcal{E}_S$ is independent of $k$ and its square root is of order $\mathcal{O}(M^{-\frac{1}{2}})$. These are similar results as Theorem (\ref{thm:gauss_init}) when applying the FGS to 2D wave equations with Gaussian initial data. A case of 2D wavefields computed by the FGS and by the spectral method is plotted in Figure \ref{fig:2d_gauss_snap}, from which one may see that with the FGS, only a small number of initial samples are needed in order to calculate a comparably accurate solution. 
\begin{figure}[htb] 
    \centering 
    \includegraphics[width=1\textwidth,trim={50 10 50 10},clip]{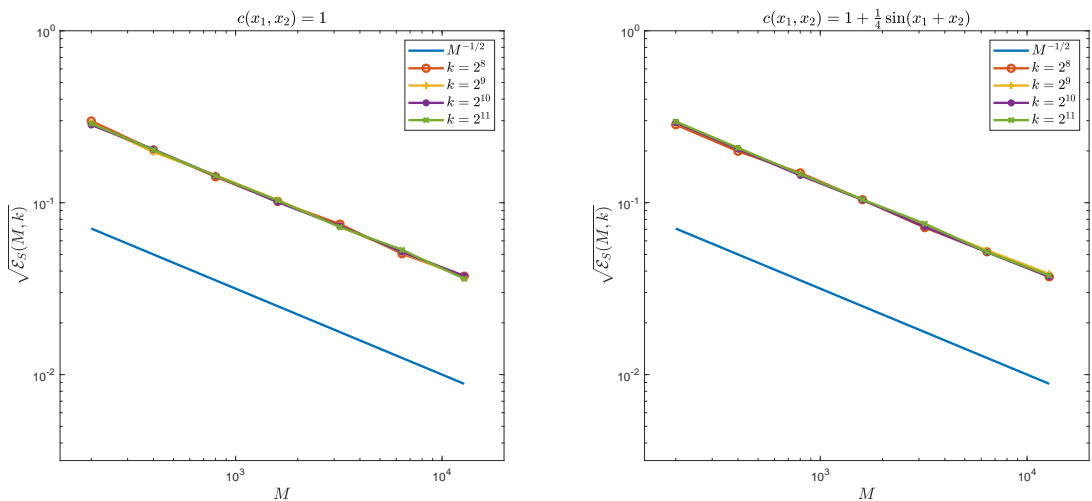} 
    \caption{The square root of the mean square sampling error $\mathcal{E}_S(M,k)$ of applying the FGS to 2D wave equations with Gaussian initial data.} 
    \label{fig:2d_gauss}
\end{figure}
\begin{figure}[htb] 
    \centering 
    \includegraphics[width=0.8\textwidth,trim={0 50 50 50},clip]{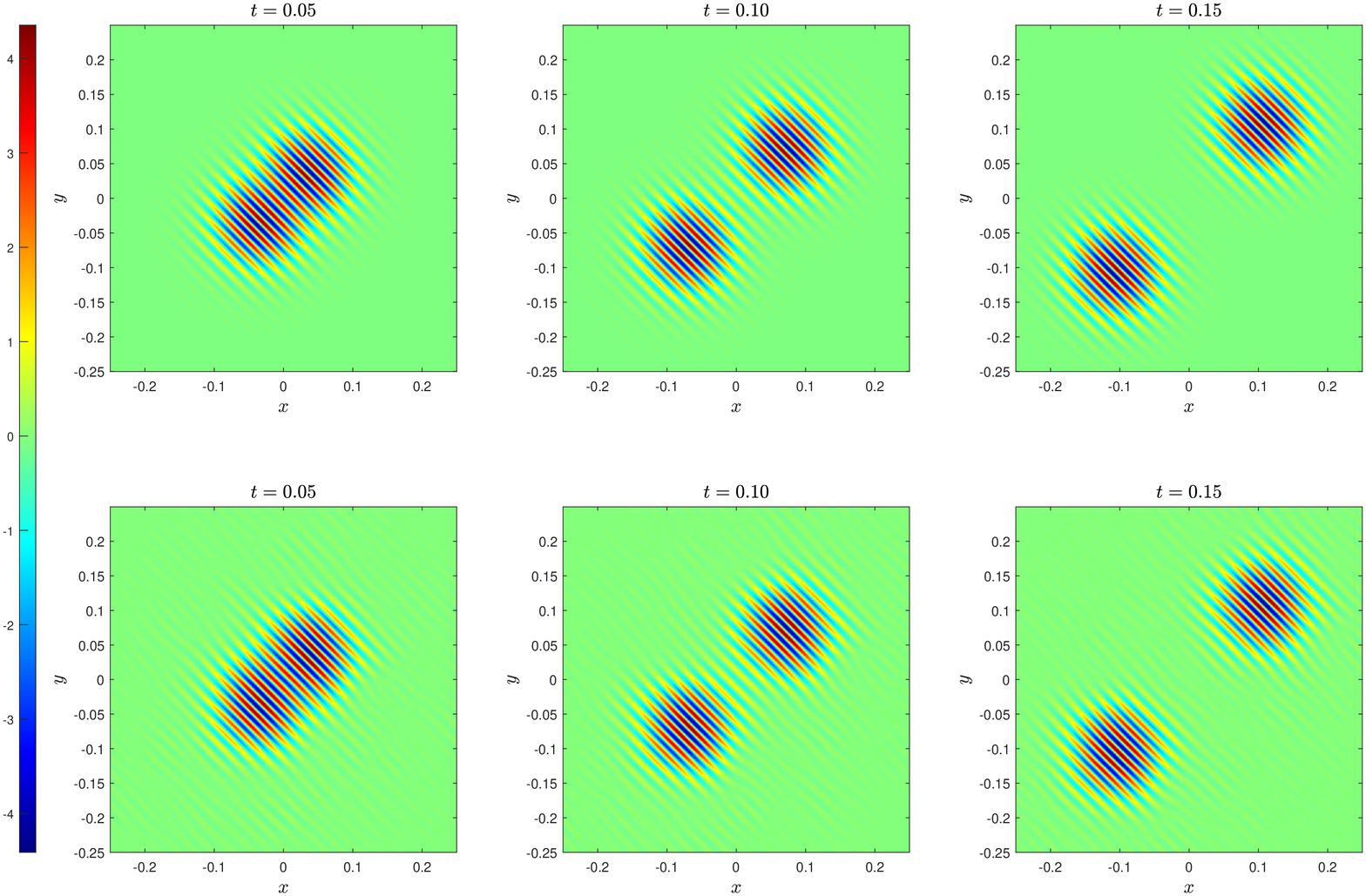} 
    \caption{The 2D wavefields computed by the FGS and the spectral method. The velocity function is $c(\bm{x})=1$ and $k=2^8$. The top row: the wave function computed by the spectral method. The bottom row: the wave function computed by the FGS using sample size $M=800$. From the left column to the right column, the wave functions in each column are at time $t=0.05,0.10,0.15$ respectively.} 
    \label{fig:2d_gauss_snap}
\end{figure}

\begin{xmpl}
Use the FGS to solve 3D wave equation. The initial conditions are
\begin{equation}
    \begin{cases}
        f_0^k(\bm{x}) = 0, \\
        f_1^k(\bm{x}) = k \left( \frac{2k}{\pi}\right)^{\frac{3}{4}} \exp\left( - \i k(x_1 + x_2 + x_3) - k \left| \bm{x} \right|^2 \right).
    \end{cases}
\end{equation}
The velocity functions are:  $c_1(\bm{x}) = 1$ and $c_2(\bm{x})=1+\frac{1}{4}\sin(x_1 + x_2 + x_3)$.
\end{xmpl}

We choose wave numbers $k = 2^5, 2^6, 2^7, 2^8$ and sample sizes $M$ = 800, 1600, 3200, 6400, 12800, 25600, 51200. Initial samples $\{z_0^{(j)}\}_{j=1}^M$ are generated from the following probability density function: 
\begin{equation}
    \pi(\bm{q},\bm{p}) = \left(\frac{k}{2\pi}\right)^3 \frac{2 \sqrt{2}}{27} \exp\left(-\frac{(p_1+1)^2 + (p_2+1)^2 + (p_3+1)^2 + 2q_1^2 + 2q_2^2 + 2q_3^2}{6/k}\right).
\end{equation}
The wave function is constructed on the spatial interval $[-1,1] \times [-1,1] \times [-1,1]$. 

We present the numerical results in Figure \ref{fig:3d_gauss}. For each velocity function, the mean square sampling error $\mathcal{E}_S$ does not depend $k$ and its square root is of order $\mathcal{O}(M^{-\frac{1}{2}})$. These results are consistent with Theorem (\ref{thm:gauss_init}) in cases of applying the FGS to 3D wave equations with Gaussian initial data. 

We also present the runtime of computing 3D FGS wave function with different sample sizes in Table \ref{tb:3d_gauss_time}. Notice that the FGS algorithm is highly parallelizable: the evolution for each initial sample $(\bm{q},\bm{p})$ and the reconstruction of wave function at each spatial location $\bm{x}$ are all independent. Therefore, one may easily conduct a parallel implementation of the FGS on either the CPU or on one or more GPUs. The data in Table \ref{tb:3d_gauss_time} is obtained by a GPU-parallel implementation. A single GPU is used in this example and it is a NVIDIA GeForce RTX 3090 with 10496 CUDA cores, 82 multiprocessors, 24 GB of global memory and compute capability 8.6. This particular graphics card is hosted by a single 16-core Intel Xeon Gold 6130 at 2.10 GHz with 22 MB of cache and 192078 GB of memory. GPU implementations are compiled using the NVIDIA HPC compiler NVFORTRAN 21.7. All experiments are run on Ubuntu 18.04. For each setting of $(M,k)$, the 3D wavefield propagates up to time $t=1$ with time step $\Delta t = 1.0\times 10^{-4}$ and is reconstructed on a mesh grid of the spatial interval $[-2,2]\times[-2,2]\times[-2,2]$ with mesh size $N = 4k$ in each dimension. The initial sampling is implemented in MATLAB and the time evolution and wavefield reconstruction are parallelized using CUDA Fortran. As Table \ref{tb:3d_gauss_time} shows, the running time of sampling and evolution is almost negligible. Besides, the runtime of evolution is insensitive to the sample sizes from $200$ to $12800$ shown in the table. The running time of reconstruction is dominant in the program and increases proportionally with the sample size. Still, one may notice that the runtime of relatively small sample sizes, like $12800$, is much smaller than that of a large sample size that is used to cover the overall phase space. This comparison is more obvious when $k$ is large, which demonstrates an advantage of the FGS that it can provide a solution of comparably high accuracy with much smaller computational cost and running time, especially in approximating high-frequency wavefields. 
\begin{figure}[htb] 
    \centering 
    \includegraphics[width=1\textwidth,trim={50 10 50 10},clip]{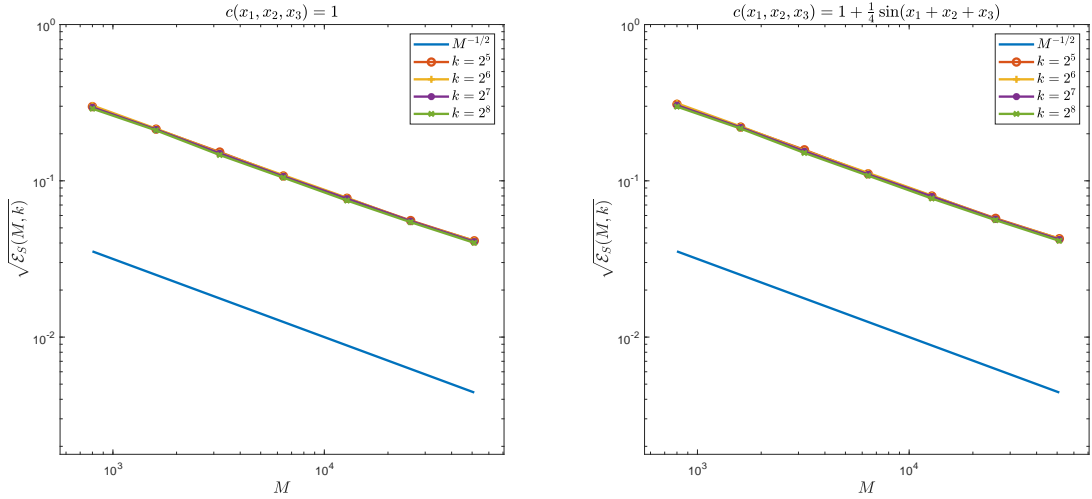} 
    \caption{The square root of the mean square sampling error $\mathcal{E}_S(M,k)$ of applying the FGS to 3D wave equations with Gaussian initial data.} 
    \label{fig:3d_gauss}
\end{figure}
\begin{table}[htb]
	\centering
    \caption{The runtime in seconds of the FGS for 3D wave equations with different sample sizes.}
	\begin{tabular}{c|rrrr|r}
		\toprule
        Initial sampling & $M = 200$ & $M=800$ & $M=3200$ & $M=12800$ & $M = 250000$\\
        \midrule
        $k=2^5$ & 0.004 & 0.009 & 0.030 & 0.094 & 1.72 \\
        $k=2^6$ & 0.005 & 0.009 & 0.027 & 0.097 & 1.79 \\
        $k=2^7$ & 0.004 & 0.010 & 0.028 & 0.099 & 1.83 \\
        $k=2^8$ & 0.003 & 0.011 & 0.032 & 0.099 & 1.88 \\
        \midrule
        Time evolution & $M = 200$ & $M=800$ & $M=3200$ & $M=12800$ & $M = 250000$\\
        \midrule
        $k=2^5$ & 0.33 & 0.34 & 0.33 & 0.33 & 3.91 \\
        $k=2^6$ & 0.31 & 0.33 & 0.33 & 0.33 & 3.91 \\
        $k=2^7$ & 0.34 & 0.35 & 0.35 & 0.35 & 3.92 \\
        $k=2^8$ & 0.35 & 0.35 & 0.34 & 0.34 & 3.91 \\
        \midrule
        Reconstruction & $M = 200$ & $M=800$ & $M=3200$ & $M=12800$ & $M = 250000$\\
        \midrule
        $k=2^5$ & 0.13 & 0.32 & 1.17 & 4.65 & 90.31 \\
        $k=2^6$ & 0.60 & 1.39 & 4.04 & 15.46 & 295.80 \\
        $k=2^7$ & 3.99 & 7.81 & 22.95 & 86.36 & 1696.41 \\
        $k=2^8$ & 33.47 & 60.00 & 173.60 & 632.50 & 12295.80\\
		\bottomrule
	\end{tabular}
    \label{tb:3d_gauss_time}
\end{table}

\subsection{Examples with WKB initial conditions}

In this subsection, we aim to study the results of the sampling error $E_S$ when applying the FGS to wave equations with WKB initial conditions. Similar to the Gaussian initial cases, we use the FGS to approximate the wave function at time $t = 0.5$ and study the relationship between the sampling error $E_S$, the wave number $k$ and the sample size $M$. To estimate the sampling error $E_S$, we do 30 simulations for each setting of sample size $M$ and wave number $k$ and calculate the square root of the mean square error, that is $\sqrt{\mathcal{E}_S}$. The FGA ansatz is numerically approximated by an FGS wave function computed by a sufficiently large sample size. For wave equations of dimensionality $d = 1,2,3$, the sample sizes used are $M_0 = 1.5 \times 10^5, 2.0 \times 10^5, 2.5 \times 10^5$ respectively. In each of the following examples, the initial sampling is based on the method given in Proposition \ref{pro:pdf_wkb}, and the fourth order Runge-Kutta method is used to evolve the ODE system (\ref{eq:ode_qp}) and (\ref{eq:ode_a}) for each initial sample up to time $t = 0.5$. The time step is chosen as $\Delta t = 1.0 \times  10^{-4}$. The FGS wave function is reconstructed based on (\ref{eq:fgs_wavefunc}). $\psi^k_{\pm}$ is computed analytically in our examples.
{
    \begin{xmpl}
        Use the FGS to solve 1D wave equation. The initial conditions are
        \begin{equation}
            \begin{cases}
                f_0^k(x) = 0, \\
                f_1^k(x) = k\left(\frac{50}{\pi}\right)^{\frac{1}{4}} \exp\left(-25\left| x \right|^2 + \frac{\i k}{2}\left(|x-0.25|^2 + |x-0.75|^2\right)\right).
            \end{cases}
        \end{equation}
        The velocity functions are:  $c_1(x) = 1$ and $c_2(x) = 1 + \frac{1}{4} \sin \left( x \right)$. 
        \end{xmpl}

We choose different wave numbers $2^9, 2^{10}, 2^{11}, 2^{12}$, and sample sizes $M$ = 50, 100, 200, 400, 800, 1600, 3200. The initial sampling is implemented as follows: first generate $M$ initial samples of $\{q^{(j)}\}_{j=1}^M$ from $\mathcal{N}(\mu_1,\Sigma_1)$, where $\mu_1 = 0$ and $\Sigma_1= \frac{1}{k} + \frac{1}{50}$. Then for each $q^{(j)}$, generate $y^{(j)}$ from $\mathcal{N}(\mu_2,\Sigma_2)$, where $\mu_2 = \frac{k q^{(j)}}{50 + k}$ and $\Sigma_2 = \frac{1}{50 + k}$. A corresponding initial sample of $p$ is obtained by $p^{(j)} = \nabla S_{\mathrm{in}}(y^{(j)}) = 2 y^{(j)} -1$. The FGS wave function is constructed on the spatial interval $[-1,1]$. 

We present the numerical results in Figure \ref{fig:1d_wkb}. For each velocity function, the square root of the mean square error $\mathcal{E}_S$ is of order $\mathcal{O}(M^{-\frac{1}{2}})$ for each $k$. When $M$ is fixed, $\sqrt{\mathcal{E}_S}$ is slightly increasing as $k$ increases, and the increase rate is of order $\mathcal{O}(k^{\frac{1}{4}})$. The dotted lines in Figure \ref{fig:1d_wkb} indicate that $\sqrt{\mathcal{E}_S}/\sqrt[4]{k}$ is independent of $k$. This example indicates that there are similar results as Theorem (\ref{thm:wkb_init}) when applying the FGS to 1D wave equations with WKB initial data.
\begin{figure}[htb] 
    \centering 
    \includegraphics[width=1\textwidth,trim={50 10 50 10},clip]{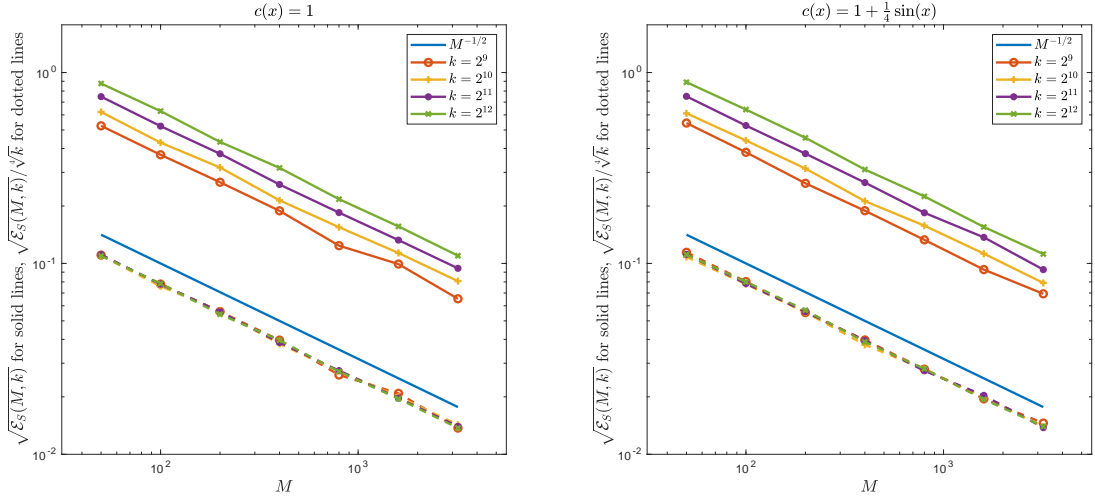} 
    \caption{The square root of the mean square sampling error $\mathcal{E}_S(M,k)$ of applying the FGS to 1D wave equations with WKB initial data.} 
    \label{fig:1d_wkb}
\end{figure}
}
\begin{xmpl}
    Use the FGS to solve 1D wave equation. The initial conditions are
    \begin{equation}
        \begin{cases}
        f_0^k(\bm{x}) = 0, \\
        f_1^k(\bm{x}) = k\frac{2}{\sqrt{3}}\pi^{-\frac{1}{4}}100^{\frac{5}{4}} x^2 \exp\left(-50\left| x \right|^2 + \i k |x-0.5|^2\right).
    \end{cases}
    \end{equation}
    The velocity functions are $c_1(\bm{x}) = 1$ and $c_2(\bm{x}) = 1 + \frac{1}{4} \sin \left( x_1 + x_2 \right)$. 
\end{xmpl}

In this example, the WKB initial condition $f^k_1$ has a non-Gaussian amplitude $a_{\mathrm{in}}(x) = \frac{1}{\mathcal{Z}_a} x^2 \exp\left(-50|x|^2\right)$, where $\mathcal{Z}_a$ is a normalization factor that ensures $\int_{-\infty}^{\infty} a^2_{\mathrm{in}}(x) \d x = 1$. Instead of using Proposition \ref{pro:pdf_wkb} (3) directly for sampling, we apply the inverse transform sampling strategy to generate initial samples $(q,p)$. First, we compute the marginal cumulative distribution function of $\mathcal{Q}$:
\begin{equation}
    F_{\mathcal{Q}}(q) = \int^{q}_{-\infty} \pi_{\mathcal{Q}}(t) \d t 
\end{equation}
where $\pi_{\mathcal{Q}}(q)$ follows from (\ref{eq:pdf_q}) and then compute the the inverse $F^{-1}_{\mathcal{Q}}$ of $F_{\mathcal{Q}}$. Then we generate samples $u_0^{(j)} (j = 1,\cdots,M)$ from the uniform distribution on $[0,1]$ and compute the corresponding samples $q^{(j)} = F^{-1}_{\mathcal{Q}}(u_0^{(j)})$ for $\mathcal{Q}$. Given each $q^{(j)}$, we implement a similar procedure to generate sample $y^{(j)}$ from the conditional distribution $\pi_{\mathcal{Y}|\mathcal{Q}}$ in (\ref{eq:pdf_y}) and finally $p^{(j)} = \nabla S_{\mathrm{in}}(y^{(j)}) = 2y^{(j)} - 1$. We construct the FGS function on $[-1,1]$. Figure \ref{fig:1d_wkb_non_gauss} shows how the sampling error $\mathcal{E}_S(M,k)$ in mean square changes with the wave number $k$ and the sample size $M$. We notice a similar trend as the preceding example on WKB initial data with Gaussian amplitude. Particularly, the almost overlapping dotted lines in the figure indicate that the increase rate of $\sqrt{\mathcal{E}_S}$ with $k$ is of order $\mathcal{O}(k^{\frac{1}{4}})$, which coincides with our conclusion in Theorem \ref{thm:wkb_init}. This example illustrates the feasibility of the FGS for WKB initial data with non-Gaussian amplitude. 

\begin{figure}[htb] 
    \centering 
    \includegraphics[width=1\textwidth,trim={50 10 50 10},clip]{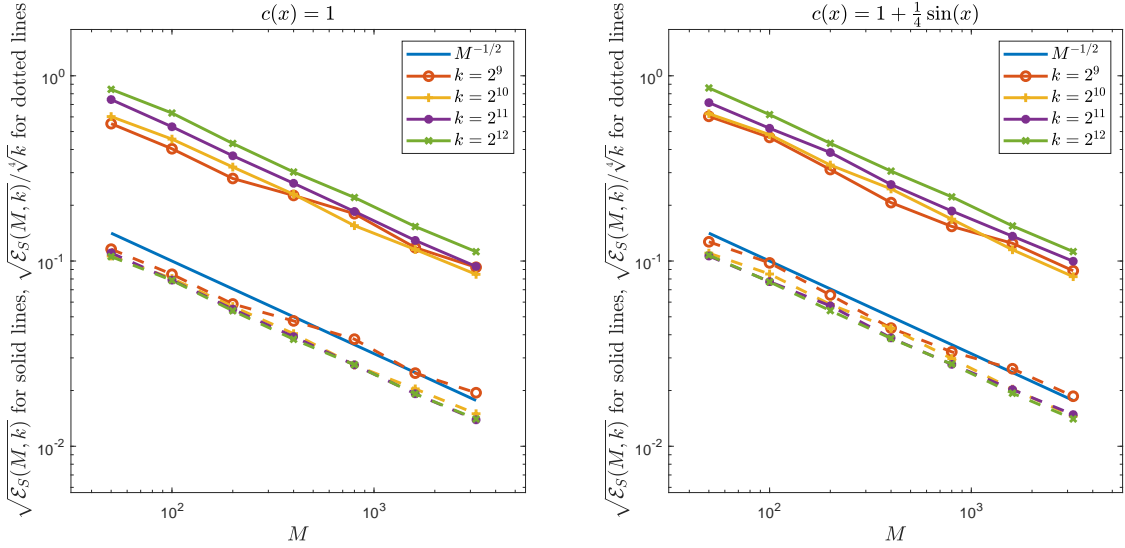} 
    \caption{The square root of the mean square sampling error $\mathcal{E}_S(M,k)$ of applying the FGS to 1D wave equations with WKB initial data with non-Gaussian amplitude.} 
    \label{fig:1d_wkb_non_gauss}
\end{figure}

\begin{xmpl}
Use the FGS to solve 2D wave equation. The initial conditions are
\begin{equation}
    \begin{cases}
        f_0^k(\bm{x}) = 0, \\
        f_1^k(\bm{x}) = k\left(\frac{50}{\pi}\right)^{\frac{1}{2}} \exp\left(-25\left| \bm{x} \right|^2 + \i k |\bm{x}-0.5|^2\right).
    \end{cases}
\end{equation}
The velocity functions are:  $c_1(\bm{x}) = 1$ and $c_2(\bm{x}) = 1 + \frac{1}{4} \sin \left( x_1 + x_2 \right)$. 
\end{xmpl}
We choose wave numbers $2^8, 2^9, 2^{10}, 2^{11}$, and sample sizes $M$ = 200, 400, 800, 1600, 3200, 6400, 12800. The initial sampling proceeds as follows: first generate $M$ initial samples of $\{\bm{q}^{(j)}\}_{j=1}^M$ from $\mathcal{N}(\mu_1,\Sigma_1)$, where
\begin{align}
    \mu_1 = \left(0,0\right), \Sigma_1 = \begin{pmatrix}
        \frac{1}{k} + \frac{1}{50} &  0 \\
        0 & \frac{1}{k} + \frac{1}{50} \\
    \end{pmatrix}.
\end{align} 
Then for each $\bm{q}^{(j)}$ generate an initial sample of $\bm{y}^{(j)}$ from $\mathcal{N}(\mu_2,\Sigma_2)$, where
\begin{align}
    \mu_2 = \left(\frac{ k q^{(j)}_1}{50 + k}, \frac{ k q^{(j)}_2}{50 + k},\right), \Sigma_2 = \begin{pmatrix}
        \frac{1}{50+k} & 0 \\
        0 & \frac{1}{50+k} \\
    \end{pmatrix}.
\end{align}
Finally compute $\bm{p}^{(j)} = \nabla S_{\mathrm{in}}(\bm{y}^{(j)}) = 2( \bm{y}^{(j)} - 0.5)$. The FGS wave function is reconstructed on the spatial interval $[-1,1] \times [-1,1]$. 

We present the numerical results in Figure \ref{fig:2d_wkb} and compare the FGS solution and the spectral method solution in Figure \ref{fig:2d_wkb_snap}. For each velocity function, the square root of $\mathcal{E}_S$ is of order $\mathcal{O}(M^{-\frac{1}{2}})$ for each $k$. When $M$ is fixed, $\sqrt{\mathcal{E}_S}$ slightly increases with $k$ at an increase rate of order $\mathcal{O}(k^{\frac{1}{2}})$. The dotted lines in Figure \ref{fig:2d_wkb} show that $\sqrt{\mathcal{E}_S}/\sqrt{k}$ is independent of $k$. This example indicates that when applying the FGS to 2D wave equations with WKB initial data one may still have similar results as Theorem (\ref{thm:wkb_init}).
\begin{figure}[htb] 
    \centering 
    \includegraphics[width=1\textwidth,trim={50 10 50 10},clip]{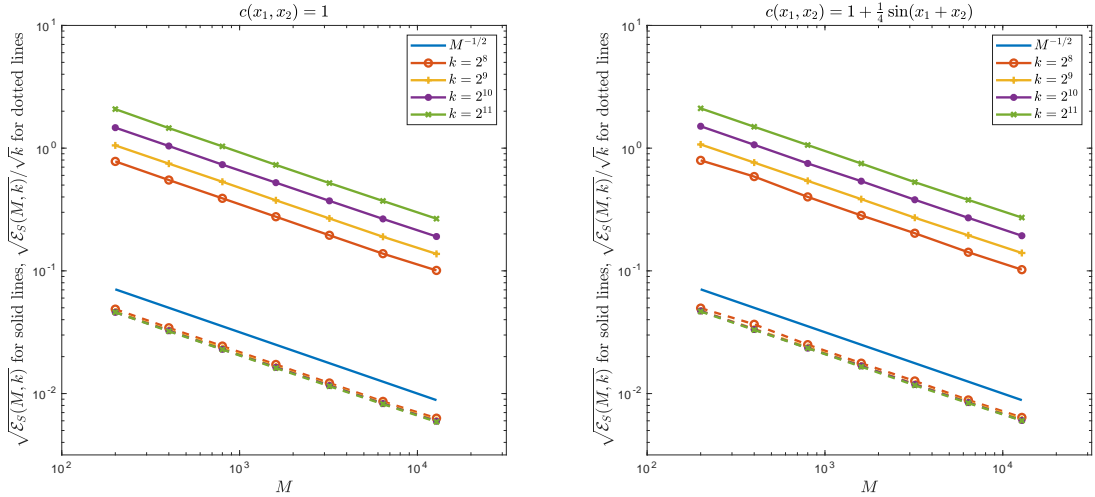} 
    \caption{The square root of the mean square sampling error $\mathcal{E}_S(M,k)$ of applying the FGS to 2D wave equations with WKB initial data.}
    \label{fig:2d_wkb}
\end{figure}
\begin{figure}[htb] 
    \centering 
    \includegraphics[width=1\textwidth,trim={0 0 50 0},clip]{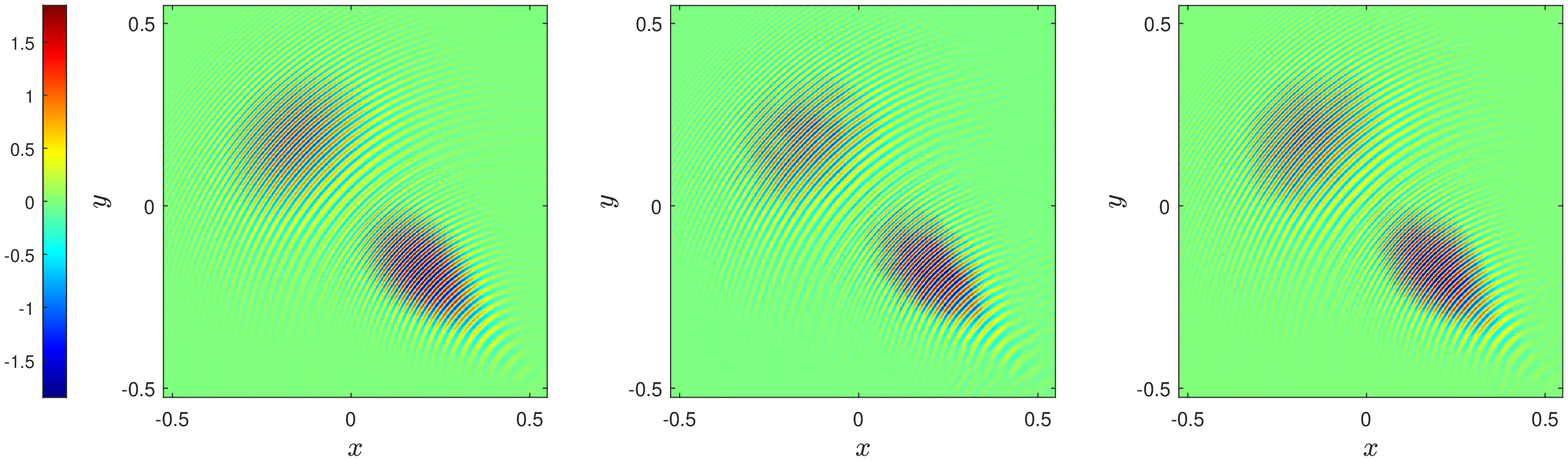} 
    \caption{The 2D wavefields computed by the FGS and the spectral method. The velocity function is $c(\bm{x}) = 1$, $k=2^8$, $t=0.25$. From the left to the right, the wave functions are the solution given by the spectral method, the solution given by the FGS with sample size $M=12800$, the solution given by the FGS with sample size $M=25600$ respectively.} 
    \label{fig:2d_wkb_snap}
\end{figure}
\begin{xmpl}
    Use the FGS to solve 3D wave equation. The initial conditions are
    \begin{equation}
        \begin{cases}
            f_0^k(\bm{x}) = 0, \\
            f_1^k(\bm{x}) = k\left(\frac{5}{\pi}\right)^{\frac{3}{4}} \exp\left(-\frac{5}{2}\left| \bm{x} \right|^2 + \i k |\bm{x}-0.5|^2\right).
        \end{cases}
    \end{equation}
    The velocity functions are:  $c_1(\bm{x}) = 1$ and $c_2(\bm{x}) = 1 + \frac{1}{4} \sin \left( x_1 + x_2 + x_3 \right)$. 
\end{xmpl}
We choose wave numbers $k = 2^5, 2^6, 2^7, 2^8$ and sample sizes $M$ = 800, 1600, 3200, 6400, 12800, 25600, 51200. The initial sampling is implemented as follows: first generate $M$ initial samples of $\{\bm{q}^{(j)}\}_{j=1}^M$ from $\mathcal{N}(\mu_1,\Sigma_1)$, where
\begin{align}
    \mu_1 = \left(0,0,0 \right), \Sigma_1 = \begin{pmatrix}
        \frac{1}{k} + \frac{1}{5} &  0  & 0 \\
        0 & \frac{1}{k} + \frac{1}{5} & 0 \\
        0 & 0 & \frac{1}{k} + \frac{1}{5}
    \end{pmatrix}.
\end{align} 
Then for each $\bm{q}^{(j)}$, generate an initial sample of $\bm{y}^{(j)}$ from $\mathcal{N}(\mu_2,\Sigma_2)$, where
\begin{align}
    \mu_2 = \left(\frac{ k q^{(j)}_1}{5 + k}, \frac{ k q^{(j)}_2}{5 + k}, \frac{ k q^{(j)}_3}{5 + k} \right), \Sigma_2 = \begin{pmatrix}
        \frac{1}{5+k} & 0 & 0 \\
        0 & \frac{1}{5+k} & 0 \\
        0 & 0 & \frac{1}{5+k}
    \end{pmatrix}.
\end{align}
A corresponding initial sample of $\bm{p}$ is obtained by $\bm{p}^{(j)} = \nabla S_{\mathrm{in}}(\bm{y}^{(j)}) = 2( \bm{y}^{(j)} - 0.5)$. The wave function is reconstructed on $[-2,2] \times [-2,2] \times [-2,2]$. 

The results are presented in Figure \ref{fig:3d_wkb}. The square root of $\mathcal{E}_S$ is of order $\mathcal{O}(M^{-\frac{1}{2}})$ for each $k$. When $M$ is fixed, $\sqrt{\mathcal{E}_S}$ increases with $k$ at a rate of order $\mathcal{O}(k^{\frac{3}{4}})$. From the dotted lines in Figure \ref{fig:3d_wkb} one may see that $\sqrt{\mathcal{E}_S}/\sqrt[4]{k^3}$ is insensitive to $k$. This example demonstrates the conclusion of Theorem (\ref{thm:wkb_init}) in the cases of applying the FGS to 3D wave equations with WKB initial data.
\begin{figure}[htb] 
    \centering 
    \includegraphics[width=1\textwidth,trim={50 10 50 10},clip]{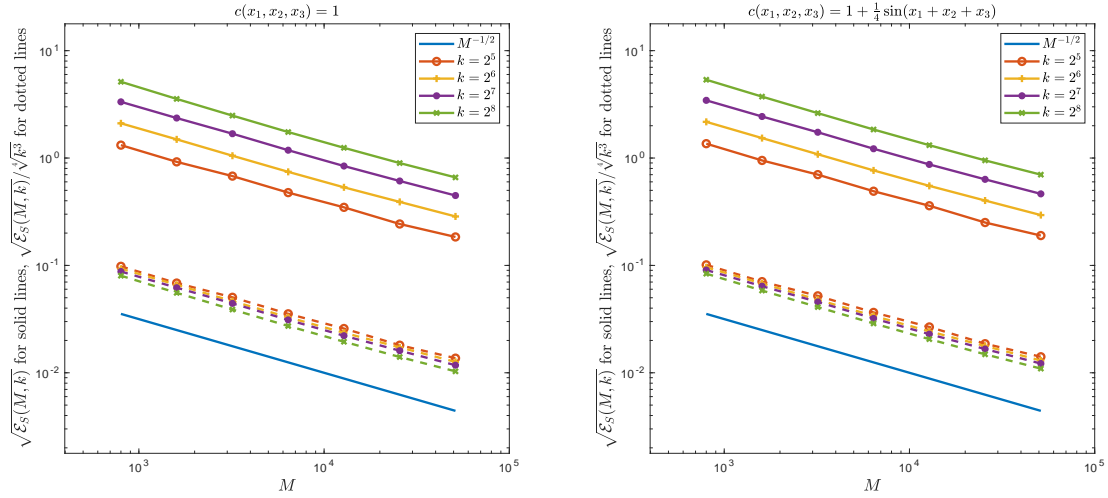} 
    \caption{The square root of the mean square sampling error $\mathcal{E}_S(M,k)$ of applying the FGS to 3D wave equations with WKB initial data.}
    \label{fig:3d_wkb}
\end{figure}

\section{Conclusion and Discussion}\label{sec:Conclusion}

In this work, we generalize the frozen Gaussian sampling (FGS) for the wave equations. The FGS shows great potential for reducing the computational costs in the propagation and reconstruction of high-frequency wavefields. The key of this stochastic simulation method is to design appropriate sampling density functions according to different types of initial conditions. An ideal sampling density should be easy to sample and at the same time, reduce the corresponding sampling error as much as possible. For wave equations, we introduce two different sampling techniques and estimate the sampling error and total error of the FGS for Gaussian and WKB initial data cases respectively. In particular, we prove and validate through numerical experiments that in Gaussian initial cases, the sampling error in the mean square sense does not depend on the asymptotic parameter $k$. In addition, we prove that in WKB initial cases, the mean square of the sampling error has an upper bound of order $\mathcal{O}(k^{\frac{d}{2}})$, where $d$ is the dimensionality. However, our numerical examples show that the proposed sampling method for the WKB cases can still achieve satisfactory approximation, which indicates that this sampling technique is able to balance sampling convenience and error reduction. It should be natural to extend the above sampling methods to the approximation of many other wave equations, such as hyperbolic systems. More advanced and accurate sampling techniques are also expected for the implementation of the FGS.

\section*{Acknowledgement}

The work of Z.Z. was supported by the National Key R\&D Program of China, Project No. 2021YFA1001200, 2020YFA0712000 and NSFC grant No. 12031013, 12171013. The work of L.C. was partially supported by the NSFC grant No. 11901601. 

	\bibliographystyle{my-plain}
	\bibliography{refs}

\end{document}